\documentclass{amsart}

\usepackage{graphicx}

\theoremstyle{plain}
\newtheorem{theorem}{Theorem}[section]
\newtheorem{lemma}{Lemma}[section]
\newtheorem{prop}{Proposition}[section]
\theoremstyle{definition}
\newtheorem{definition}{Definition}[section]

\newtheorem{example}{Example}[section]

\begin{document}

\title[Laplacians on smooth distributions as $C^*$-algebra multipliers]{Laplacians on smooth distributions \\ as $C^*$-algebra multipliers}

\author{Yuri A. Kordyukov}
 
\address{Institute of mathematics, Ufa scientific center, Russian Academy of Sciences}
 
\email{yurikor@matem.anrb.ru}
\thanks{Supported by RFBR (grant No.~16-01-00312-a).}

\keywords{foliation, Hilbert module, Laplacian, hypoelliptic operators, smooth distribution, multiplier}

\subjclass{58J60, 53C17, 46L08, 58B34}

\begin{abstract}
In this paper we continue the study of spectral properties of Laplacians associated with an arbitrary smooth distribution on a compact manifold, initiated in a previous paper. Under assumption that the singular foliation generated by the distribution is smooth, we prove that the Laplacian associated with the distribution defines an unbounded regular self-adjoint operator in some Hilbert module over the foliation $C^*$-algebra. 
\end{abstract}

%\translator{T.R.Anslator}

\maketitle

%\tableofcontents

\section{Introduction}\label{s:intro}
In this paper we continue the study of spectral properties of Laplacians associated with an arbitrary smooth distribution on a compact manifold, initiated in \cite{hypo}. First, we recall their definition. 

Let $M$ be a connected compact smooth manifold of dimension $n$ equipped with a smooth positive density $\mu$. Let $H$ be a smooth distribution on $M$ of rank $d$ (that is, $H$ is a smooth subbundle of the tangent bundle $TM$ of $M$) and $g$ a smooth fiberwise inner product on $H$. We define the horizontal differential $d_Hf$ of a function 
$f\in C^\infty(M)$ as the restriction of its differential $df$ to
$H\subset TM$. Thus, $d_Hf$ is a section of the dual bundle $H^*$ of
$H$, $d_Hf\in C^\infty(M,H^*)$, and we get a well-defined first order differential operator $d_H : C^\infty(M) \to C^\infty(M,H^*)$. The Riemannian metric $g$ and the positive smooth density $\mu$ induce inner products on $C^\infty(M)$ and $C^\infty(M,H^*)$, that allows us to consider the adjoint operator $d^*_H : C^\infty(M,H^*)\to C^\infty(M)$ of $d_H$. The Laplacian $\Delta_H$ associated with $(H,g,\mu)$ is the second order differential operator on $C^\infty(M)$ given by
\[
\Delta_H=d^*_Hd_H.
\]
Let $X_1,\ldots, X_d$ be a local orthonormal frame in $H$ defined on an open subset $\Omega\subset M$. It can be easily checked that the restriction of the operator $\Delta_H$ to $\Omega$ is given by
\[
\Delta_H\left|_\Omega\right.=\sum_{j=1}^d X^*_jX_j.
\]
The operator $\Delta_H$ can be also defined by means of the associated quadratic form
\[
(\Delta_Hu,u)=\int_M |d_Hu(x)|_g^2\,d\mu(x), \quad u\in C^\infty(M).
\]

We will consider the Laplacian $\Delta_H$ as an unbounded linear operator in the Hilbert space $L^2(M,\mu)$ with initial domain $C^\infty(M)$.

\begin{theorem}[\cite{hypo}]\label{t:ss}
The operator $\Delta_H$ is an essentially self-adjoint operator in $L^2(M,\mu)$ (that is, its closure is a self-adjoint operator).
\end{theorem} 
 
Theorem~\ref{t:ss} is proved, using a well-known result by Chernoff \cite{Chernoff} based on some facts from theory of linear symmetric first order hyperbolic systems, in particular, using in an essential way the fact of finite propagation speed of wave solutions of such equations.

Let us consider some examples.

\begin{example} The case when $H=TM$ and $g$ is a Riemannian metric on $M$, $g=\sum_{i,j}g_{ij}(x)dx_idx_j$, corresponds to a structure of Riemannian manifold on $M$. In this case, there is a canonical choice of $\mu$. The density $\mu$ can be taken to be the Riemannian volume form of $g$: $\mu =\sqrt{\det g}dx_1\ldots dx_n$. Then $\Delta_H$ coincides with the Laplace-Beltrami operator of the Riemannian metric $g$:
\[
\Delta_H=-\frac{1}{\sqrt{\det g}}\sum_{i,j=1}^n\frac{\partial}{\partial x_i}\left( g^{ij}\sqrt{\det g}\frac{\partial}{\partial x_j}\right).
\]
The choice of an arbitrary smooth positive density $\mu =\mu(x) dx_1\ldots dx_n$ corresponds to a structure of weighted manifold on $M$. In this case, $\Delta_H$ is the corresponding weighted Laplacian (cf., for instance, \cite{Grigoryan}):
\[
\Delta_H=-\frac{1}{\mu}\sum_{i,j=1}^n\frac{\partial}{\partial x_i}\left( g^{ij}\mu\frac{\partial}{\partial x_j}\right).
\]
\end{example} 
 
Consider the space $C^\infty(M,TM)$ of smooth vector fields on $M$. We recall that it has a natural structure of $C^\infty(M)$-module: 
\[
a\in C^\infty(M), X\in C^\infty(M,TM)\mapsto a\cdot X\in C^\infty(M,TM),
\]
and a structure of Lie algebra with respect to the Lie bracket operation of vector fields, which is defined for vector fields $X=\sum_iX_i\frac{\partial}{\partial x_i}$ and $Y=\sum_iY_i\frac{\partial}{\partial x_i}$ by
\[
[X,Y]=XY-YX=\sum_i\left(\sum_j X_j\frac{\partial Y_i}{\partial x_j}-\sum_j Y_j\frac{\partial X_i}{\partial x_j}\right)\frac{\partial}{\partial x_i}.
\]
For the distribution $H$, we consider the subspace $C^\infty(M,H)$ of $C^\infty(M,TM)$, which consists of smooth vector fields on $M$, tangent to $H$ at each point: $X(m)\in H_m$ for any $m\in M$.

\begin{example} 
Suppose that the distribution $H$ is completely integrable. We recall that this means that, for any $X,Y\in C^\infty(M,H)$, their Lie bracket $[X,Y]$ is in $C^\infty(M,H)$. Then the classical Frobenius theorem states that there exists a smooth foliation $\mathcal F$ on $M$ such that $H$ coincides with the tangent bundle of $\mathcal F$: $H=T\mathcal F$. In this case, a smooth metric $g$ in fibers of $H$ gives rise to a smooth Riemannian metric along the leaves of $\mathcal F$. In any foliated coordinate neighborhood with coordinates $(x,y)$, $x\in \mathbb R^d$, $y\in \mathbb R^{n-d}$, it has the form
\[
g=\sum_{i,j=1}^dg_{ij}(x,y)dx_idx_j.
\]
In this case, there is no canonical choice of a smooth positive density $\mu$. The operator $\Delta_H$ is the associated leafwise Laplace operator. In a foliated coordinate neighborhood with coordinates $(x,y)$, $x\in \mathbb R^d$, $y\in \mathbb R^{n-d}$, it has the form
\[
\Delta_H=-\frac{1}{\mu}\sum_{i,j=1}^d\frac{\partial}{\partial x_i}\left( g^{ij}(x,y)\mu(x,y)\frac{\partial}{\partial x_j}\right), \quad x\in \mathbb R^d, \quad y\in \mathbb R^{n-d}.
\]
This operator is degenerate elliptic (more precisely, leafwise elliptic) second order differential operator. We observe that it doesn't coincide, in general, with the leafwise differential operator on $M$, given by the family of Laplace-Beltrami operators of Riemannian metrics along the leaves.  

In \cite{Kord95}, the author constructed classes $\tilde{\Psi}^{m,l}(M, {\mathcal  F})$ of anisotropic pseudodifferential operators of different orders in tangential and transverse directions and the corresponding scale of anisotropic Sobolev spaces $H^{s,k}(M,{\mathcal  F})$. This allows us to construct a parametrix (in some sense) for the leafwise Laplacian and, more generally, for an arbitrary longitudinally elliptic operator and investigate various aspects of functional calculus for such operators. 

In \cite{Connes79,Connes80}, Connes developed methods of noncommutative geometry to study various problems of geometry and analysis on foliated manifolds. They are based on constructions of such objects as the holonomy groupoid of the foliation, operator algebras associated with the foliation, longitudinal pseudodifferential calculus (see \cite{survey,umn-survey} for further information). 

One of advantages of using the methods of theory of operator algebras and noncommutative geometry for the study of longitudinally elliptic operators is that they allow one to consider operators, defined by the same formal differential expression in essentially different Hilbert spaces (in different representations) and compare spectral properties of such operators. For instance, the leafwise Laplacian $\Delta_H$ can be also considered as the family of $\{\Delta_L : L\in M/\mathcal F\}$ elliptic differential operators along the leaves of the foliation $\mathcal F$. Such applications rely on the fact that the Laplacian $\Delta_H$ (and, moreover, any formally self-adjoint longitudinally elliptic operator) gives rise to an unbounded regular operator $\Delta$ in some Hilbert module over the $C^*$-algebra $C^*(M,\mathcal F)$. The operator $\Delta_H$ and the family $\{\Delta_L : L\in M/\mathcal F\}$ are the images of $\Delta$ under the corresponding representations of the $C^*$-algebra $C^*(M,\mathcal F)$. In \cite{Vassout}, these results were extended to the case of elliptic differential operators on Lie groupoids. 

The notion of unbounded regular operator on a Hilbert module over a  $C^*$-algebra was introduced by Baaj in his thesis \cite{Baaj80} (see also  \cite{Baaj-Julg}), called an unbounded multiplier over a $C^*$-algebra. Regular operators have many nice properties, similar to the properties of closed densely defined operators in a Hilbert space. In particular, for self-adjoint regular operators, there is a continuous functional calculus (see, for instance, \cite{Lance}). 
\end{example}

\begin{example}  
Another important class of distributions is given by completely non-integrable distributions (or ``bracket generating'', ``satisfying H\"ormander  condition''). We recall that a distribution $H$ is called completely non-integrable, if vector fields $X\in C^\infty(M,H)$ and their iterated Lie brackets $[X_1,X_2]$, $[X_1,[X_2,X_3]]$ and so on, span the $C^\infty(M)$-module $C^\infty(M,TM)$. In this case, the set $(M, H, g, \mu)$ introduced above is called a structure of sub-Riemannian manifold. We also observe that in the case in question there is a canonical choice of a smooth positive density $\mu$ given by the Popp measure. For more information on sub-Riemannian manifolds, see, for instance, \cite{Montgomery}.

By the classical H\"ormander's theorem on sum of squares \cite{Hormander67}, a Laplacian $\Delta_H$ associated with a completely non-integrable distibution is a hypoelliptic operator. First of all, it satisfies so-called subelliptic estimates. There exists $\epsilon>0$ such that, for any $s\in \mathbb R$,
\[
\|u\|_{s+\epsilon}^2 \leq
C_s\left(\|\Delta_Hu\|_s^2+\|u\|_s^2\right), \quad u\in C^\infty(M),
\]
where $C_s>0$ is some constant and $\|\cdot\|_s$ denotes a norm in the Sobolev space $H^s(M)$. The hypoellipticity property is that, if $u\in C^{-\infty}(M)$ and $\Delta_Hu\left|_U\right.\in H^s(U)$ for some $s\in \mathbb R$ and some domain $U\subset M$, then $u \in H^{s+\varepsilon}(U)$. In particular, if $u\in C^{-\infty}(M)$ and $\Delta_Hu\left|_U\right.\in C^{\infty}(U)$, then $u \in C^{\infty}(U)$. 
\end{example} 

The study of spectral properties of the operator $\Delta_H$ associated with an arbitrary distribution $(H,g,\mu)$ was started by the author in  \cite{hypo}. The main observation of \cite{hypo} is that, under rather week assumptions on $H$, the operator $\Delta_H$ can be considered as a longitudinally hypoelliptic operator with respect to some foliation associated with $H$. But this foliation is, in general, singular. Here, following  \cite{Andr-Skandalis-I}, by a singular foliation $\mathcal F$ on a smooth manifold $M$, we mean a locally finitely generated $C^\infty_c(M)$-submodule of the $C^\infty_c(M)$-module of compactly supported smooth vector fields $C^\infty_c(M, TM)$, invariant under the Lie brackets. 
A submodule $\mathcal E \subset C^\infty_c(M, TM)$ is said to be locally finitely generated, if, for any $p\in M$, there exist an open neighborhood  $U$ of $p$ in $M$ and vector fields $X_1,\ldots, X_k\in \mathcal E\left|_U\right. \subset C^\infty_c(U,TU)$ such that, for any $f\in C^\infty_c(U)$ and $X\in \mathcal E$, the relation $fX\left|_U\right.=\sum_{j=1}^kf_jX_j$ holds with some $f_1,\ldots, f_k\in C^\infty_c(M)$. 

Any regular foliation on a smooth manifold $M$, understood as a partition of  $M$ into disjoint union of immersed submanifolds (leaves), gives rise to a singular foliation. Here, as a module $\mathcal F$, one should take the space of compactly supported smooth vector fields on $M$, tangent to  leaves of the foliation at each point of $M$. We remark that the submodule  $\mathcal F\subset C^\infty_c(M, TM)$, defining a structure of singular foliation, also gives rise to aa partition of  $M$ into disjoint union of immersed submanifolds of nonconstant dimension (\cite{Stefan, Sussmann}). But, in general, there could be several different modules $\mathcal F$, defining the same partition into leaves.

Any smooth distribution $H$ on a compact manifold $M$ determines a singular foliation $\mathcal F_H$ as follows. Recall that $C^\infty(M,H)$ denotes the submodule of $C^\infty(M,TM)$, which consists of smooth vector fields, tangent to $H$ at every point. Let $\mathcal F_H$ be the minimal submodule in $C^\infty(M,TM)$, which contains $C^\infty(M,H)$ and is invariant under the Lie brackets. Assume that $\mathcal F_H$ is finitely generated as a $C^\infty(M)$-module, that is, there exist $X_1,\ldots, X_N\in \mathcal F_H$ such that, for any $X\in \mathcal F_H$, the representation $X=f_1X_1+\ldots+f_NX_N$ holds with some $f_i\in C^\infty(M)$. Then $\mathcal F_H$ is a singular foliation in the above sense. 

In \cite{Andr-Skandalis-I,Andr-Skandalis-II}, Androulidakis and Skandalis extended to singular foliations the methods of noncommutative geometry of regular foliations, developed by Connes in \cite{Connes79}. In particular, they constructed classes of pseudodifferential operators and associated scale of Sobolev Hilbert modules $H^k$ over the algebra $C^*(M,\mathcal F)$. They also proved that a formally self-adjoint longitudinally elliptic operator defines an unbounded regular multiplier $C^*(M,\mathcal F)$ with corresponding applications in spectral theory (see also \cite{Andr14}).

In \cite{hypo}, we use the methods developed in \cite{Andr-Skandalis-I,Andr-Skandalis-II} as well as the methods of the study of hypoelliptic H\"ormander operators of sum of squares type to get some basic results on properties of the Laplacian $\Delta_H$ for an arbitrary smooth distribution.   
First of all, in \cite{hypo}, we prove subelliptic estimates for the operator $\Delta_H$. 

\begin{theorem}[\cite{hypo}]\label{t:Hs-hypo}
There exists $\epsilon>0$ such that, for any $s\in \mathbb R$, we have 
\[
\|u\|_{s+\epsilon}^2 \leq
C_s\left(\|\Delta_Hu\|_s^2+\|u\|_s^2\right), \quad u\in C^\infty(M),
\]
where $C_s>0$ is some constant and $\|\cdot\|_s$ denotes a norm in the leafwise Sobolev space $H^s(\mathcal F_H)$.
\end{theorem}

As a consequence, we immediately obtain the following result on longitudinal hypoellipticity. 

\begin{theorem}[\cite{hypo}]\label{t:hypo}
If $u\in H^{-\infty}(\mathcal F_H):=\bigcup_{t\in \mathbb R}H^t(\mathcal F_H)$ such that $\Delta_Hu\in H^s(\mathcal F_H)$ for some $s\in \mathbb R$, then $u \in H^{s+\varepsilon}(\mathcal F_H)$. 
\end{theorem}

These results, in particular, allow us to give another proof of essential self-adjointness of the operator $\Delta_H$ and also prove that, for any function  $\varphi$ from the Schwartz space $\mathcal S(\mathbb R)$, the operator $\varphi(\Delta_H)$ is leafwise smoothing with respect to the foliation $\mathcal F_H$, that is, it extends to a bounded operator from $H^s(\mathcal F_H)$ to $H^t(\mathcal F_H)$ for any $s,t\in \mathbb R$.

We note that, unlike the paper \cite{Andr-Skandalis-II} that deals with Hilbert modules over the $C^*$-algebra $C^*(M,\mathcal F_H)$, we work in \cite{hypo} with the concrete representation of the $C^*$-algebra $C^*(M,\mathcal F_H)$ on $L^2(M,\mu)$. This enables us to use some results of theory of linear operators in a Hilbert space.   

The goal of this paper is to construct and study unbounded self-adjoint operators on Hilbert modules over the $C^*$-algebra $C^*(M,\mathcal F_H)$ determined by the Laplacian $\Delta_H$. The main result of the paper is stated in the following theorem (more precise statements are given in Theorems~\ref{t:multi1} and \ref{t:multi2} below).

\begin{theorem}\label{t:multi}
Suppose that the foliation $\mathcal F_H$ associated with a smooth distribution $H$ is a regular foliation.

(1) There exists a second order differential operator $\Delta$ on the holonomy groupoid $G$ of the foliation $\mathcal F_H$, whose image under the natural representation of $C^*(M,\mathcal F_H)$ in the Hilbert space $L^2(M,\mu)$ coincides with $\Delta_H$. 

(2) Consider the $C^*$-algebra $C^*(M,\mathcal F_H)$ as a Hilbert module over itself and $\Delta$ as an unbounded linear operator in  $C^*(M,\mathcal F_H)$ with domain $C^\infty_c(G)$. Then the closure $\overline \Delta$ of $\Delta_H$ is an unbounded regular self-adjoint operator.
\end{theorem}

Consider a natural representation of $C^*(M,\mathcal F_H)$ in the Hilbert space $L^2(\tilde L)$ of square integrable functions on the holonomy covering $\tilde L$ of an arbitrary leaf $L\in M/\mathcal F_H$. Under this representation the operator $\Delta$ maps in some elliptic differential operator $\Delta_{\tilde L}$, which, in general, does not coincide with the restriction of $\Delta_H$ to $\tilde L$. As an straightforward consequence of Theorem \ref{t:multi} we obtain in a standard way (cf., for instance, \cite{Kord95}) the following statement.  

\begin{theorem}\label{t:spectrum}
(1) The spectrum $\sigma(\Delta_H)$ of $\Delta_H$ in $L^{2}(M,\mu)$ contains its leafwise spectrum
\[
\sigma_{\mathcal  F}(\Delta_H)=\overline{\bigcup \{\sigma(\Delta_{\tilde L}): L \in M/\mathcal F\}},
\]
where $\sigma(\Delta_{\tilde L})$ is the spectrum of $\Delta_{\tilde L}$ in $L^{2}(\tilde L)$.

(2) It the holonomy groupoid is amenable (that is, $C^*(M,\mathcal F)\cong C^*_r(M,\mathcal F)$), then $\sigma(\Delta_H)$ coincides with $\sigma_{\mathcal  F}(\Delta_H)$.
\end{theorem} 

For the leafwise Laplacian on a smooth foliation $\mathcal F$, this result was proved in \cite{Kord95}.

The paper is organized as follows. Section \ref{s:multi} begins with summary of necessary information from foliation theory. Then we construct a second order differential operator $\Delta$ on the holonomy groupoid of  $\mathcal F_H$, proving teh first part of Theorem \ref{t:multi}. After a brief information on regular operators in Hilbert modules, we give a refined statement of the second part of Theorem~\ref{t:multi}. The proof of this theorem is based on a parametrix construction for the operator $\Delta$ and consists of several steps. First, we show how to prove the tehorem, using an appropriate parametrix for $\Delta$. Then we turn to the problem of constructing such a parametrix and reduce it to a local problem of constructing a smooth family of parametrices for a smooth family of hypoelliptic H\"ormander sum of squares type operators. The latter problem is solved in Section \ref{s:hypo-family}. Here we use a classical parametrix construction for a hypoelliptic H\"ormander operator given in the paper of Rothschild and Stein \cite{Rothschild-Stein}. Actually, we need to prove that, if the operator depends smoothly on a parameter, then its parametrix constructed in \cite{Rothschild-Stein} also depends smoothly on the parameter. This is result is a very natural generalization of the results of Rothschild and Stein, but, unfortunately, we didn't managed to find it in literature. Therefore, Section \ref{s:hypo-family} contains the proof of this result. For this, we survey main notions and constructions of \cite{Rothschild-Stein}, checking in addition smooth dependence of all constructed objects on the parameter.   
 
\section{Operators in Hilbert modules}\label{s:multi}
\subsection{Preliminaries on foliation theory} In this section we give some necessary information on foliation theory  (for more details on noncommutative geometry of foliations, see \cite{survey,umn-survey} and references therein). 

Recall that a smooth manifold $M$ is equipped with a (regular) foliation $\mathcal F$ of dimension $p$, if $M$ is represented as a disjoint union $M=\cup_{\lambda \in M/\mathcal  F}L_\lambda$ of a family of connected, (bijectively) immersed submanifolds of dimension $p$ in such a way that, for an arbitrary point $m$ of $M$, there exists a coordinate neighborhood $\Omega\cong U\times T$, where $U$ is an open subset of $\mathbb R^p$, $T$ is an open subset of $\mathbb R^q$, such that, for any leaf $L$ of $\mathcal F$, the components of $L\cap \Omega$ have the form $U\times \{y\}$ for some $y\in T$. The neighborhood $\Omega$ will be called a foliated coordinate neighborhood.

The foliation $\mathcal F$ defines a subbundle $F=T{\mathcal F}$ of the tangent bundle $TM$, called the tangent bundle of $\mathcal F$. It consists of all vectors, tangent to leaves of $\mathcal F$. The subspace $C^\infty(M,F)$ of vector fields on $M$ tangent to leaves of $\mathcal  F$ at every point of $M$ is a singular foliation in the sense of the definition given in Introduction.

Recall that a set $G$ carries a structure of groupoid with the set
of units $G^{(0)}$, if there are defined maps
\begin{enumerate}
\item $\Delta : G^{(0)}\rightarrow G$ (the diagonal map or the unit map);
\item an involution $i:G\rightarrow G$ called the inversion and
written as $i(\gamma)=\gamma^{-1}$;
\item the range map $r:G\rightarrow G^{(0)}$ and the source map $s:G\rightarrow G^{(0)}$;
\item an associative multiplication $m: (\gamma,\gamma')\rightarrow
\gamma\gamma'$ defined on the set
\[
G^{(2)}=\{(\gamma,\gamma')\in G\times G : r(\gamma')=s(\gamma)\},
\]
\end{enumerate}
satisfying the conditions
\begin{itemize}
\item[(i)] $r(\Delta(x))=s(\Delta(x))=x$ and $\gamma\Delta(s(\gamma))=\gamma$,
$\Delta(r(\gamma))\gamma=\gamma$;
\item[(ii)]~$r(\gamma^{-1})=s(\gamma)$ and $\gamma\gamma^{-1}=\Delta(r(\gamma))$.
\end{itemize}
It is convenient to represent an element $\gamma\in G$ as an arrow $\gamma :x \to y$, where $x=s(\gamma)$ and $y=r(\gamma)$. We will also use standard notation (for $x,y\in G^{(0)}$):
\begin{gather*}
G^x=\{\gamma\in G:r(\gamma)=x\} =r^{-1}(x),\quad G_x=\{\gamma\in
G:s(\gamma)=x\} =s^{-1}(x),\\
G^x_y=\{\gamma\in G : s(\gamma)=x, r(\gamma)=y\}.
\end{gather*}

The holonomy groupoid $G=G(M,{\mathcal F})$ of a foliated manifold
$(M,{\mathcal F})$ is defined in the following way. Let $\sim_h$ be
the equivalence relation on the set of piecewise-smooth leafwise paths
$\gamma:[0,1]\rightarrow M$, setting $\gamma_1\sim_h \gamma_2$, if
$\gamma_1$ and $\gamma_2$ have the same initial and final points and
the same holonomy maps: $h_{\gamma_1} = h_{\gamma_2}$. The holonomy groupoid $G$ is the set of $\sim_h$-equivalence classes of piecewise-smooth leafwise paths. The set of units $G^{(0)}$ is a manifold $M$. The multiplication in $G$ is given by the product of paths. The
corresponding range and source maps $s,r:G\rightarrow M$ are given
by $s(\gamma)=\gamma(0)$ and $r(\gamma)=\gamma(1)$. Finally, the
diagonal map $\Delta:M\rightarrow G$ takes any $x\in M$ to the
element in $G$ given by the constant path $\gamma(t)=x, t\in [0,1]$.
To simplify the notation, we will identify $x\in M$ with
$\Delta(x)\in G$.

For any $x\in M$, the map $s$ maps $G^x$ onto the leaf $L_x$ through
$x$. The group $G^x_x$ coincides with the holonomy group of $L_x$.
The map $s:G^x\rightarrow L_x$ is the regular covering with the
covering group $G^x_x$, called the holonomy covering.

The holonomy groupoid $G$ is a Lie groupoid. This means that it can be equipped with a structure of smooth manifold (in general, non-Hausdorff and non-paracompact) of dimension $2p+q$ such that the sets $G^{(0)}$, $G$ and $G^{(2)}$ are smooth manifolds, the maps $r$, $s$, $i$ and $m$ are smooth maps, the maps $r$ and $s$ are submersions and the map $\Delta$ is an embedding.

Let $\alpha \in C^\infty(M,|T{\mathcal F}|)$ be an arbitrary smooth positive leafwise density on $M$. For any $x\in M$, define a smooth positive density $\nu ^{x}$ on $G^x$ as the lift of the density $\alpha $ by the holonomy covering map $s:G^x\to M$. The family $\{\nu^x:x\in M\}$ is a smooth Haar system on $G$. 

Introduce a structure of involutive algebra on
$C^{\infty}_c(G)$ by
\begin{align*}
k_1\ast k_2(\gamma)&=\int_{G^x} k_1(\gamma_1)
k_2(\gamma^{-1}_1\gamma)\,d\nu^x(\gamma_1),\quad \gamma\in G^x,\\
k^*(\gamma)&=\overline{k(\gamma^{-1})}, \quad \gamma\in G.
\end{align*}
For any $x\in M$, there is a natural representation of
$C^{\infty}_c(G)$ in the Hilbert space $L^2(G^x,\nu^x)$ given, for
$k\in C^{\infty}_c(G)$ and $\zeta \in L^2(G^x,\nu^x)$, by
\begin{equation}\label{e:Rx}
R_x(k)\zeta(\gamma)=\int_{G^x}k(\gamma^{-1}\gamma_1)
\zeta(\gamma_1) d\nu^x(\gamma_1),\quad r(\gamma)=x.
\end{equation}
The completion of the involutive algebra $C^{\infty}_c(G)$ in the
norm
\[
\|k\|=\sup_x\|R_x(k)\|
\]
is called the reduced $C^*$-algebra of the groupoid $G$ and denoted
by $C^{\ast}_r(G)$. There is also defined the full
$C^{\ast}$-algebra of the groupoid $C^{\ast}(G)$, which is the
completion of $C^{\infty}_c(G)$ in the norm
\[
\|k\|_{\text{max}}=\sup \|\pi(k)\|,
\]
where supremum is taken over the set of all $\ast$-representations
$\pi$ of the algebra $C^{\infty}_c(G)$ in Hilbert spaces.

\subsection{Construction of the operator on a Hilbert module}
Let $H$ be a smooth distribution on a compact manifold $M$. It defines a singular foliation $\mathcal F_H$ as follows. Recall that $C^\infty(M,H)$ denotes the submodule of $C^\infty(M,TM)$, which consists of smooth vector fields, tangent to $H$ at every point of $M$. Let $\mathcal F_H$ be the minimal submodule of $C^\infty(M,TM)$, which contains $C^\infty(M,H)$ and is invariant under Lie brackets. If it is finitely generated, it defines a singular foliation.

From now on, we will assume that $\mathcal F_H$ is a regular foliation. This means that $\mathcal F_H$ coincides with the subspace of smooth vector fields on $M$, tangent to leaves of some smooth foliation, which will be also denoted by $\mathcal F_H$. It is easy to see that $H \subset T\mathcal F_H$ and the restriction $H_L$ of $H$ to any leaf $L$ of $\mathcal F_H$ is completely non-integrable. The Laplacian $\Delta_H$ associated with a Riemannian metric $g$ on $H$ and a smooth positive density $\mu$ on $M$ is a longitudinal differential operator with respect to the foliation $\mathcal F_H$. Note that the restriction of $\Delta_H$ to a leaf $L$ of $\mathcal F_H$, in general, doesn't coincide with the Laplacian associated with the distribution $H_L$ for some choice of a smooth positive density $\alpha$ on $L$.   

As above, let $\mu$ be a smooth positive density on $M$. Following \cite{renault,Fack-Skandalis}, we define a natural representation $R_\mu$ of the $C^*$-algebra $C^*(M,\mathcal F_H)$ in the Hilbert space $L^2(M,\mu)$. Let $\{\nu^x=s^*\alpha : x\in M\}$ be the smooth Haar system on the holonomy groupoid $G$ of $\mathcal F_H$, defined by a smooth positive leafwise density $\alpha\in C^\infty(M,|T{\mathcal F}|)$. There exists a smooth non-vanishing function $\delta$ on $G$ such that, for any $f\in C_c(G)$, the following equality holds:
\[
\int_M\left(\int_{G^x}\delta(\gamma) f(\gamma^{-1})d\nu^x(\gamma)\right)d\mu(x)=\int_M\left(\int_{G^x}f(\gamma)d\nu^x(\gamma)\right)d\mu(x).
\]
In terminology of \cite{renault,Fack-Skandalis}, the function $\delta$ defines a homomorphism of the groupoid $G$ in the multiplicative group $\mathbb R_+$, and the measure $\mu$ on $M$ is a quasi-invariant measure of module $\delta$. Without loss of generality, we may assume that $\delta(x)=1$ for any $x\in M\subset G$.

\begin{definition}
A regularizing operator $R_\mu(k) : L^2(M,\mu)\to L^2(M,\mu)$ associated with a leafwise kernel $k\in C^\infty_c(G)$ is defined for $u\in L^2(M,\mu)$ by
\[
R_\mu(k)u(x)=\int_{G^x} k(\gamma)\delta^{-\frac 12}(\gamma) u(s(\gamma))d\nu^x(\gamma), \quad x\in M.
\]
\end{definition}

Let $\phi: \Omega\stackrel{\cong}{\to} U\times T$ and $\phi': \Omega'\stackrel{\cong}{\to} U^\prime \times T$ be two compatible foliated charts on $M$ and $W(\phi,\phi')\subset G \stackrel{\cong}{\to} U\times U^\prime \times T$ the corresponding coordinate chart on $G$ \cite{Connes79} (see also \cite{survey,umn-survey}). The restrictions of the maps $r : W(\phi,\phi') \to \Omega$ and $s : W(\phi,\phi') \to \Omega^\prime$ of the holonomy groupoid $G$ to $W(\phi,\phi')$ are given by
\[
r(x,x^\prime,y)=(x,y),\quad s(x,x^\prime,y)=(x^\prime,y), \quad (x,x^\prime,y)\in U\times U^\prime \times T.
\]
In the charts $\phi$ and $\phi^\prime$, the density $\mu$ is written as $\mu=\mu(x,y)|dx||dy|$ and $\mu=\mu^\prime(x^\prime,y^\prime)|dx^\prime||dy^\prime|$, respectively, and the density $\alpha$ as $\alpha=\alpha(x,y)|dx|$ and $\alpha=\alpha^\prime(x^\prime,y^\prime)|dx^\prime|$, respectively. Then $\delta\in C^\infty(U\times U\times T)$ is given by (see \cite[Proposition VIII.12]{Connes79})
\[
\delta(x,x^\prime,y)=\frac{\mu(x,y)\alpha^\prime(x^\prime,y)}{\mu^\prime(x^\prime,y)\alpha(x,y)}, \quad (x,x^\prime,y)\in U\times U^\prime \times T.
\]
For any kernel $k$ with support in $W$, $k\in C^\infty_c(W)\cong C^\infty_c(U\times U^\prime \times T)$, the operator $R_\mu(k) : C^\infty(\Omega^\prime)\to C^\infty(\Omega)$ has the form
\[
R_\mu(k)f(x,y)=\int k(x,x^\prime,y)\left(\frac{\mu^\prime(x^\prime,y)}{\mu(x,y)}\right)^{1/2}(\alpha(x,y))^{1/2}(\alpha^\prime(x^\prime,y))^{1/2} f(x^\prime,y)dx^\prime.
\]

The following statement is a refined formulation of the first part of Theorem~\ref{t:multi}. 

\begin{theorem}\label{t:multi1}
There exists a second order differential operator $\Delta$ on $G$ such that its image under the representation $R_\mu$ coincides with $\Delta_H$, that is, for any $k\in C^\infty_c(G)$,  
\begin{equation}\label{e:LXDelta}
R_\mu(\Delta k)=\Delta_H R_\mu(k).
\end{equation} 
\end{theorem}

\begin{proof}
We show that, for any vector field $X\in C^\infty(M,T\mathcal F_H)$ considered as a first order differential operator in $C^\infty(M)$, there exists a first order differential operator $L_X: C^\infty_c(G)\to C^\infty_c(G)$ such that
\begin{equation}\label{e:LXk}
XR_\mu(k)=R_\mu(L_Xk), \quad k\in C^\infty_c(G).
\end{equation} 
It is easy to see that there exists a unique vector field $s^*X$ (resp. $r^*X$) on $G$ such that $ds_\gamma(s^* X(\gamma)) =X(s(\gamma))$ and $dr_\gamma(s^* X(\gamma))=0$ (resp. $ds_\gamma(r^* X(\gamma)) =0$ and $dr_\gamma(s^* X(\gamma))=X(r(\gamma))$) for any $\gamma\in G$. If we write $X$ in the foliated coordinate neighborhood $\Omega\cong U\times T$ on $M$ as $X_{(x,y)}=\sum_{j=1}^p X^j(x,y)\frac{\partial}{\partial x_j}$, then in the coordinate neighborhood $W\subset G \cong U\times U^\prime \times T$ determined by a pair of compatible foliated charts $\Omega\cong U\times T$ and $ \Omega'\cong U^\prime \times T$ we have
\[
s^*X_{(x,x^\prime,y)}=\sum_{j=1}^p X^j(x^\prime,y)\frac{\partial}{\partial x^\prime_j}, \quad r^*X_{(x,x^\prime,y)}=\sum_{j=1}^p X^j(x,y)\frac{\partial}{\partial x_j}.
\] 
In local coordinates we get
\begin{align*}
XR_\mu(k)f(x,y)=& \sum_{j=1}^p \int X^j(x,y)\frac{\partial}{\partial x_j} \left[K(x,x^\prime,y)\left(\frac{\mu^\prime(x^\prime,y)}{\mu(x,y)}\right)^{1/2}(\alpha(x,y))^{1/2}(\alpha^\prime(x^\prime,y))^{1/2}\right]
f(x^\prime,y)dx^\prime\\
=& R_\mu(L_Xk)f(x,y),
\end{align*}
where
\[
L_XK(x,x^\prime,y)=\sum_{j=1}^p X^j(x,y)\left[\frac{\partial}{\partial x_j} K(x,x^\prime,y)-\frac{1}{2}\frac{\partial (\mu/\alpha)}{\partial x_j}(x,y)K(x,x^\prime,y)\right].
\]
From here, one can easily see that the formula \eqref{e:LXk} holds with some operator of the form
\[
L_X=r^*X+r^*l_X, \quad l_X\in C^\infty(M).
\]
Since $X^*=-X+c_X$ for some $c_X\in C^\infty(M)$, we also get
\[
X^*R_\mu(k)=-XR_\mu(k)+c_XR_\mu(k)=R_\mu((-L_X+r^*c_X)k).
\]
Thus, we obtain that
\begin{equation}\label{e:LXk1}
X^*R_\mu(k)=R_\mu(\tilde L_Xk),
\end{equation}
where
\[
\tilde L_X=-r^*X+r^*\tilde l_X,  \quad \tilde l_X\in C^\infty(M).
\] 

Let $M=\bigcup_{\alpha=1}^M \Omega_\alpha$ be a finite open cover of  $M$ such that, for any $\alpha=1,\ldots,M$, there exists a local orthonormal basis $X^{(\alpha)}_1,\ldots, X^{(\alpha)}_d\in C^\infty(\Omega_\alpha, H\left|_{\Omega_\alpha}\right.)$. As mentioned above, the restriction of $\Delta_H$ to $\Omega_\alpha$ is written as
\[
\Delta_H\left|_{\Omega_\alpha}\right.=\sum_{j=1}^d(X^{(\alpha)}_j)^* X^{(\alpha)}_j.
\]
Let $\phi_\alpha\in C^\infty(M)$ be a partition of unity, subordinate to the cover, ${\rm supp}\,\phi_\alpha\subset U_\alpha$, and
$\psi_\alpha\in C^\infty(M)$ be such that ${\rm
supp}\,\psi_\alpha\subset U_\alpha$, $\phi_\alpha\psi_\alpha=\phi_\alpha$. Then we have
\[
\Delta_H=\sum_{\alpha=1}^M \phi_\alpha (\Delta_H\left|_{\Omega_\alpha}\right.)\psi_\alpha=\sum_{\alpha=1}^M\sum_{j=1}^d \phi_\alpha (X^{(\alpha)}_j)^* X^{(\alpha)}_j\psi_\alpha.
\]
Set
\begin{equation}\label{e:Delta}
\Delta=\sum_{\alpha=1}^M\sum_{j=1}^d  r^*\phi_\alpha\tilde L_{X^{(\alpha)}_j} L_{X^{(\alpha)}_j}r^*\psi_\alpha.
\end{equation}
Using \eqref{e:LXk} and \eqref{e:LXk1}, it is easy to establish the relation \eqref{e:LXDelta}.  
\end{proof}

\subsection{Regular operators in Hilbert modules and parametrix}\label{s:regular}
Recall that a Hilbert module over a $C^*$-algebra $A$ is a right $A$-module $E$ endowed with a positive definite sesquilinear map $\langle \cdot, \cdot\rangle : E\times E\to A$ such that the $\|x\|_E=\|\langle x,x\rangle \|_A^{1/2}$ equips $E$ with a structure of Banach space. 

Let us consider a Hilbert module $E$ over the $C^*$-algebra $C^*(M,\mathcal F_H)$ defined as follows. $E=C^*(M,\mathcal F_H)$, the structure of right module is given by the right multiplication by elements of the algebra and the inner product has the form $\langle a,b\rangle=a^*b$.   
We will consider the operator $\Delta$ as an unbounded, densely defined operator on the Hilbert module $E$ with domain $\mathcal A=C^\infty_c(G)$. This operator is formally self-adjoint in the sense that, for any $k_1,k_2\in C^\infty_c(G)$, the following equality holds: 
\begin{equation}\label{e:ss}
\langle \Delta k_1,k_2\rangle=\langle k_1,\Delta k_2\rangle.
\end{equation}
Indeed, by \eqref{e:LXDelta}, we have
$R_\mu(\Delta k)=\Delta_H R_\mu(k)$. Taking adjoints and using the fact that  $R_\mu$ is a $\ast$-representation, we obtain that $R_\mu((\Delta k)^*)= R_\mu(k^*)\Delta_H$. Let us check that  
\[
R_\mu((\Delta k_1)^*\ast k_2)=R_\mu(k^*_1\ast \Delta k_2).
\]  
Indeed, using the fact that $R_\mu$ is a $\ast$-representation, we have:
\[
R_\mu((\Delta k_1)^*\ast k_2)=R_\mu((\Delta k_1)^*)R_\mu(k_2)=R_\mu(k_1^*)\Delta_H R_\mu(k_2),
\]
\[
R_\mu(k^*_1\ast \Delta k_2)=R_\mu(k^*_1) R_\mu(\Delta k_2)=R_\mu(k_1^*)\Delta_H R_\mu(k_2).
\]   
Since $R_\mu$ is injective on $C^\infty_c(G)$, this implies the equality
\[
(\Delta k_1)^*\ast k_2=k^*_1\ast \Delta k_2,
\] 
equivalent to \eqref{e:ss}.

Since the domain of $\Delta$ is dense, its formal self-adjointness immediately implies the existence of the closure $\overline{\Delta}$. 
 
Recall that an unbounded operator $T$ on a Hilbert module $E$ is called regular, if it is densely defined, its adjoint is densely defined, and its graph admits an orthogonal complement, which means that $A\oplus A = G \oplus G^\bot$, where $G = \{(x, Tx) : x \in \operatorname{Dom} T\}$ is the graph of $T$ and $G^\bot = \{(T^*y,-y) :  y \in \operatorname{Dom} T^*\}$ is its orthogonal complement with respect to an obvious $A$-valued inner product on $A\oplus A$.

The next statement is a refined formulation of the second part of Theorem~\ref{t:multi}. 

\begin{theorem}\label{t:multi2} 
The operator $\overline{\Delta}$ gives rise to an unbounded regular self-adjoint operator on the Hilbert module $E$.
\end{theorem}

The proof of this theorem is based on a construction of a parametrix for the operator $\Delta$. 

\begin{theorem}\label{t:param}
There are elements $Q$, $R$ and $S$ of $C^*(M,\mathcal F_H)$, considered as morphisms of the Hilbert module $E$, such that the following identities hold (on $\mathcal A$):  
\begin{equation}\label{e:param}
I-Q\Delta = R, \quad I-\Delta Q = S.
\end{equation}
Moreover, the operators $\Delta R$ and $\Delta S^*$ extend to compact morphisms of the Hilbert module $E$ and, therefore, belong to the algebra  $C^*(M,\mathcal F_H)$.
\end{theorem}

The proof of this theorem will be given in Section \ref{s:local}. In this section, we demonstrate how Theorem \ref{t:param} implies Theorem~\ref{t:multi2}. Here, actually, we will repeat with more details a similar proof, given in \cite{Vassout} (see also \cite{Andr-Skandalis-II}). 

\begin{lemma}\label{l:DR}
The following relations hold:

(a)
$\overline{\Delta}\, \overline{Q}= \overline{\Delta Q}$, $\overline{\Delta}\, \overline{R}= \overline{\Delta R}$.

(b) $ \operatorname{Dom} \overline{\Delta} = \operatorname{Im}  \overline{Q}+ \operatorname{Im} \overline{R}$.
\end{lemma}

\begin{proof}
(a)  Let $u_n\in \operatorname{Dom} (\overline{\Delta}\, \overline{R})$ be a sequence such that $u_n\to u$ and the sequence $\overline{\Delta}\, \overline{R} u_n$ is convergent. Then, for the sequence $v_n=\overline{R} u_n$, we obtain that $v_n\to \overline{R} u$ and the sequence $\overline{\Delta}v_n$ is convergent. Therefore, $\overline{R} u\in \operatorname{Dom} (\overline{\Delta})$ and $u\in \operatorname{Dom} (\overline{\Delta}\, \overline{R})$. Thus, the operator $\overline{\Delta}\, \overline{R}$ is closed. Since $\Delta R\subset \overline{\Delta}\, \overline{R}$, this immediately implies that $\overline{\Delta R}\subset \overline{\Delta}\, \overline{R}$. By Theorem~\ref{t:param}, the operator $\Delta R$ extends to a morphism $\overline{\Delta R}$ of the Hilbert module $E$. Hence, it is everywhere defined and $\overline{\Delta R}=\overline{\Delta}\, \overline{R}$. 

(b) First of all, we observe that (a) immediately implies that $\operatorname{Im}  \overline{Q}+ \operatorname{Im} \overline{R}\subset \operatorname{Dom} \overline{\Delta}$. 

Let us prove the reverse inclusion. Let $u\in \operatorname{Dom} \overline{\Delta}$. Thus, there exists a sequence $u_n\in C^\infty_c(G)$ such that $u_n$ converges in the norm of $C^*(M,\mathcal F_H)$ to $u$ and $\Delta u_n$ converges in the norm to $\overline{\Delta}u$. By \eqref{e:param}, we have
\[
u_n=Q\Delta u_n + R u_n.
\]
Passing to the limit in this equality, we obtain that
\[
u=\overline{Q}\, \overline{\Delta} u + \overline{R} u,
\]
and, therefore, $u\in \operatorname{Im} \overline{Q}+ \operatorname{Im} \overline{R}$. 
\end{proof}

\begin{proof}[Proof of Theorem~\ref{t:multi2}]
First, we prove that $\Delta^*=\overline{\Delta}$. Observe that, since $\Delta$ is formally self-adjoint, $\overline{\Delta}\subset \overline{\Delta}^*=\Delta^*$. It remains to prove that $\operatorname{Dom} \Delta^*\subset \operatorname{Dom} \overline{\Delta}$. Since $\Delta Q = I + S$, we have $(\Delta Q)^* = I + S^*$. Note that the following relation holds: $Q^*\Delta^* \subset (\Delta Q)^*$, which means that, if $u\in \operatorname{Dom} Q^*\Delta^*=\operatorname{Dom} \Delta^*$, then $Q^*\Delta^*u= (\Delta Q)^*u$. Indeed, since $\Delta Q$ is a bounded morphism and $u\in \operatorname{Dom}\Delta^*$, for any $v\in C^*(M,\mathcal F)$, we have
\[
\langle (\Delta Q)^*u,v\rangle = \langle u, \Delta Q v\rangle = \langle \Delta^* u, Q v\rangle = \langle Q^*\Delta^* u, v\rangle.  
\]

Therefore, for any $u\in \operatorname{Dom} \Delta^*$, we have the relation $u =Q^*\Delta^*u-S^*u$, which implies that $\operatorname{Dom} \Delta^*\subset \operatorname{Im} Q^*+\operatorname{Im} S^*$. 

Passing to adjoints in the second relation in \eqref{e:param}, we get the following identity (as operators on $\mathcal A$): 
\[
I-Q^*\Delta = S^*. 
\]
One can apply to it the assertion of Lemma \ref{l:DR} (a), which allows us to claim that $\operatorname{Im}  Q^*+ \operatorname{Im} S^*\subset \operatorname{Dom} \overline{\Delta}$ and, therefore, $\operatorname{Dom} \overline\Delta = \operatorname{Im} Q^*+\operatorname{Im} S^*$.

It remains to prove that the graph $G(\overline{\Delta})=\{(u,\overline\Delta u), u\in \operatorname{Dom} \overline\Delta\}$ of the morphism $\overline \Delta$ admits an orthogonal complement in $C^*(M,\mathcal F)\oplus C^*(M,\mathcal F)$. Using Lemma \ref{l:DR}, one can check that  $G(\overline{\Delta})$ is described as follows:
\[
G(\overline{\Delta})=\{(\overline{Q}x+\overline{R}y, \overline{\Delta Q}x+\overline{\Delta R}y), (x,y)\in E\times E\}. 
\]

Consider the operator $T$ in $E\oplus E$ given by the matrix
\[
T=
\begin{bmatrix}
\overline{Q} & \overline{R}\\
\overline{\Delta Q} & \overline{\Delta R}
\end{bmatrix}.
\]
By Lemma \ref{l:DR} (a), it is a bounded operator. Moreover, since the elements of the matrix, defining $T$, belong to the algebra $C^*(M,\mathcal F_H)$, the operator $T$ is a morphism of the Hilbert module $E\oplus E$. 

Hence, the assertion of the theorem follows from the following well-known fact (see, for instance, \cite[Theorem 2.3.3]{Manuilov-Troitski}). If $T$ is a morphism of a Hilbert module $E$ and $\operatorname{Im} T$ is closed, then $\operatorname{Im} T$ admits an orthogonal complement. Indeed, in this case, the following relation holds: $\operatorname{Im} T\oplus \operatorname{Ker} T^* = E$. 
\end{proof}

\subsection{Localization of the problem}\label{s:local}
In this section, we show how to reduce the proof of Theorem~\ref{t:param} to some local statement for families of differential operators.

First of all, we recall the notion of $G$-operator, introduced in \cite{Connes79}. For any $\gamma\in G$, $\gamma : x\to y$, define the left translation operator $L(\gamma) : C^\infty(G^x)\to C^\infty(G^x)$ by
\[
L(\gamma)f(\gamma^\prime)=f(\gamma^{-1}\gamma^\prime), \quad \gamma^\prime \in G^y.
\]
Note that the operator $L(\gamma)$ extends to an isometric operator from the Hilbert space $L^2(G^x,\nu^x)$ to $L^2(G^y,\nu^y)$. A $G$-operator is any family $\{P_x, x\in M\}$, where $P_x$ is a linear continuous map in $C^\infty(G^x)$, which is left-invariant: $L(\gamma)\circ P_x=P_y\circ L(\gamma)$ for any $\gamma : x\to y$. An example of a $G$-operator is given by an operator of the form $P_x=R_x(k), x\in M$, where $k\in C^\infty_c(G)$ and $R_x$ is the representation of the algebra $C^{\infty}_c(G)$ in $L^2(G^x,\nu^x)$ given by \eqref{e:Rx}.

Let us compute the image of $\Delta$ under the representation $R_x$, that is, a differential $G$-operator $\{\Delta_x: x\in M\}$ such that, for any $k\in C^\infty_c(G)$, the following identity holds: 
\[
R_x(\Delta k)=\Delta_x R_x(k).
\] 
For any $k\in C^\infty(G)$, define a function $\tilde k \in C^\infty(G)$ by $\tilde k(\gamma)=k(\gamma^{-1}), \gamma\in G$. It is easy to check that, for any vector field $X\in C^\infty(M,T\mathcal F)$ and any function $a\in C^\infty(M)$, we have the identities:
\[
\widetilde{(r^*X)k}=(s^*X)\tilde k, \quad \widetilde{(r^*a)k}=(s^*a)\tilde k, \quad k\in C^\infty(G).
\]
Using these identities, from \eqref{e:Delta} and \eqref{e:Rx}, we get 
\[
\Delta_x=\sum_{\alpha=1}^d\sum_{j=1}^p  s^*\phi_\alpha\tilde R_{X^{(\alpha)}_j} R_{X^{(\alpha)}_j}s^*\psi_\alpha,
\]
where
\[
R_X=s^*X+s^*l_X, \quad \tilde R_X=-s^*X+s^*\tilde l_X.
\] 
It is this operator that was denoted by $\Delta_{\tilde L}$ in Introduction in Theorem \ref{t:spectrum} (with $\tilde L=G^x$). 

Consider the longitudinally elliptic operator $\Delta_M$ on $M$ given by
\[
\Delta_M=\sum_{\alpha=1}^M\sum_{j=1}^d  \phi_\alpha (-{X^{(\alpha)}_j}+\tilde l_{X^{(\alpha)}_j}) (X^{(\alpha)}_j+l_{X^{(\alpha)}_j}) \psi_\alpha.
\]
If we restrict this operator to the leaf $L_x$ through $x\in M$ and then lift it to the holonomy covering $G^x$ by use of the map $s:G^x\to L_x$, then we get the operator $\Delta_x$. Remark that the operator $\Delta_M$, in general, doesn't coincides with $\Delta_H$. 

Let $\Omega\cong U\times T$, $U\subset \mathbb R^p$,  $T\subset \mathbb R^q$, be an arbitrary foliated coordinate neighborhood, where there is defined a local orthonormal frame $X_j, j=1,\ldots,d$ in $H$. Since the vector fields $X_j$ are tangent to the foliation $\mathcal F_H$, they are tangent to plagues $U\times \{y\}$ in $\Omega$, and, therefore, $X_j(x,y)\in \mathbb R^p\cong \mathbb R^p\oplus \{0\}\subset \mathbb R^p\oplus \mathbb R^q$ for any $(x,y)\in U\times T$. In particular, any $X_j$ is given by a family $\{X_{j,y}, y\in T\}$ of vector fields on $U$. For any function $a\in C^\infty(U\times T)$, we will denote by $a_y\in C^\infty(U\times \{y\})\cong C^\infty(U)$ its restriction to $U\times \{y\}, y\in T$.

It is easy to check that the restriction of the operator $\Delta_M$ to $\Omega$ is given by a smooth family $\{\Delta_y, y\in T\}$ of second order differential operator on $U$ of the form
\begin{equation}\label{e:Delta-local}
\Delta_y=-\sum_{j=1}^d X^2_{j,y}+\sum_{j=1}^da_{j,y}X_{j,y}+b_y, \quad y\in T, 
\end{equation}
where $a_j,b\in C^\infty(U\times T)$.

A crucial role in the proof of Theorem~\ref{t:param} is played by the following fact. 

\begin{theorem}\label{t:param-local}
For any $m\in M$, there exists a foliated coordinate neighborhood $\Omega\cong U\times T$, $U\subset \mathbb R^p$,  $T\subset \mathbb R^q$, such that, for any $\phi\in C^\infty_c(\Omega)$, there exists a family $\{Q_y, y\in T\}$ of compact operators $L^2(U)$, continuous in the uniform operator topology, with the Schwartz kernel compactly supported in $U\times U\times T$, such that
\[
Q_y\Delta_y=\phi_y I-R_y,\quad \Delta_y Q_y=\phi_y I-S_y,\quad y\in T,
\]
where the operators $R_y$, $S_y$, $\Delta_y R_y$ and $\Delta_y S^*_y$ on $C^\infty(U)$ extend to compact operators in $L^2(U)$, depending continuously on $y\in T$ in the uniform operator topology.
\end{theorem}

The proof of Theorem~\ref{t:param-local} will be given in Section
 \ref{s:hypo-family}. Here we complete the proof of Theorem~\ref{t:param}, using Theorem~\ref{t:param-local}. 

The following construction given in \cite{Connes79} (see, in particular,  \cite[Proposition VIII.7b)]{Connes79}) allows one to construct $G$-operators, starting from a continuous family of integral operators, defined in a foliated coordinate neighborhood. Let $\Omega\cong U\times T$ be a foliated coordinate neighborhood and $\{P_y : y\in T\}$ be a continuous family of operators in $C^\infty(U)\cong C^\infty(U\times \{y\})$. Assume that the Schwarz kernel $k_y\in C^{-\infty}(U\times U)$ is compactly supported in $U\times U\times T$. (In this case, we will say that the family  $\{P_y\}$ is compactly supported in $U\times T$.) A natural embedding of $W=U\times U\times T$ into $G$ allows one to consider $k$ as a distribution on $G$ and, therefore, define the corresponding $G$-operator $\{P^\prime_x : x\in M\}$. This operator can be also described as follows. For an arbitrary $x\in M$, we define an equivalence relation on $G^x\cap s^{-1}(\Omega)$, setting $\gamma_1\sim \gamma_2$, if $\gamma_1^{-1}\gamma_2\in W$. Each equivalence class $\ell$ is open (since $W$ is open) and connected (since $U$ is connected). Therefore, $\ell$ is a connected component and $G^x\cap s^{-1}(\Omega)=\cup \ell $ is the representation of $G^x\cap s^{-1}(\Omega)$ as the union of connected components. The restriction of $s$ to $\ell$ is a homeomorphism of $\ell$ on some plague $s(\ell)=U\times\{y\}$ of $\Omega$. The kernel  $K_x(\gamma,\gamma_1)$ of the operator $P^\prime_x : C^\infty(G^x)\to C^\infty(G^x)$ may be different from zero only if $\gamma$ and $\gamma_1$ are in the same component of $G^x\cap s^{-1}(\Omega)$, and the restriction of the operator to the component $\ell$ corresponds under the map $s$ to an operator $P_y$, acting on $s(\ell)=U\times\{y\}$. 

By \cite[Proposition VIII.7b)]{Connes79}, if for any $y\in T$ the operator $P_y$ is compact in $L^2(U)$ and the map $y\to P_y$ is continuous in the uniform operator norm in $L^2(U)$, then the corresponding $G$-operator $P^\prime$ belongs to $C^*(M,\mathcal F_H)$. 

The proof of Theorem~\ref{t:param} is completed by means of the standard gluing construction of local parametrices (cf. \cite[Proposition IX.2)]{Connes79}).

\begin{proof}[Proof of Theorem~\ref{t:param}]
Let $M=\cup_i\Omega_i$ be a finite covering of $M$ by foliated coordinate neighborhoods, for each of which the statement of Theorem~\ref{t:param-local} holds, $\Omega_i\cong U_i\times T_i$, $\phi_i$ be a partition of unity, subordinate to this covering, $\psi_i\in C^\infty_c(\Omega_i)$ be functions such that $\psi_i=1$ on the support of $\phi_i$. Let $\{\Delta_{i,y} : y\in T_i\}$ be a smooth family of differential operators on  $U_i$ defined by \eqref{e:Delta-local} in a foliated coordinate neighborhood $\Omega_i\cong U_i\times T_i$. Observe that the $G$-operator $(\psi_i\Delta_i)^\prime$ obtained by use of the above construction from the family $\{\psi_{i,y}\Delta_{i,y}: y\in T_i\}$, coincides with the differential $G$-operator $\{s^*\psi\Delta_x: x\in M\}$. By Theorem~\ref{t:param-local}, for any $i$, there is a continuous family $\{Q_{i,y} : y\in T_i\}$ of compact operators in $L^2(U_i)$ with compact support with $U_i\times T_i$ such that 
\[
Q_{i,y}\Delta_{i,y}=\phi_{i,y} I-R_{i,y},\quad \Delta_{i,y} Q_{i,y}=\phi_{i,y} I-S_{i,y},\quad y\in T_i,
\]
where the operators $R_{i,y}$, $S_{i,y}$, $\Delta_{i,y} R_{i,y}$ and $\Delta_{i,y} S^*_{i,y}$ on $C^\infty(U_i)$ extend to compact operators in $L^2(U_i)$, continuously depending on $y\in T$ in the uniform operator topology. It is easy to check that the $G$-operator $Q=\sum_i Q^\prime_{l,i}(\psi_i\circ s)$ is a desired one. 
\end{proof}

\section{Families of hypoelliptic operators and their parametrices}\label{s:hypo-family}
This section is devoted to the proof of Theorem~\ref{t:param-local}. The main part of the proof of Theorem~\ref{t:param-local} is the construction of a smooth family of parametrices for a smooth family of hypoelliptic H\"ormander sum of squares type operators. We will follow the classical parametrix construction given in the paper of Rothschild and Stein \cite{Rothschild-Stein}, checking in the process smooth dependence of all constructed objects on the parameter. 
 
\subsection{Free systems of vector fields}
A fundamental step in the paper \cite{Rothschild-Stein} is a construction of the lift of a given system of vector fields, satisfying H\"ormander's condition, to a systems of vector fields in a space of larger dimension, which is free of order $m$. Therefore, we begin with necessary preliminaries on free systems of vector field, following \cite{Rothschild-Stein,Hoermander-Melin}. 

Let $X_1,\ldots, X_d$ be a family of smooth vector fields defined in a domain $U\subset \mathbb R^p$. Denote by $\mathcal X(U)$ the space of smooth vector fields in $U$. For any $X\in \mathcal X(U)$, we define the linear operator $\operatorname{ad} X$ in $\mathcal X(U)$, which takes any $Y\in \mathcal X(U)$ to the vector field $[X, Y]$. For any set $I=(i_1,\ldots, i_k)$ of $k=|I|$ integers in the interval between $1$ and $d$, we define the corresponding basic commutator $X_{[I]}\in \mathcal X(U)$ by
\[
X_{[I]} = \operatorname{ad} X_{i_1}\ldots \operatorname{ad} X_{i_{k-1}} X_{i_k}.
\]
For any $k\geq 1$ and $x\in U$, we denote by $H^k_x$ the subspace of $\mathbb R^p$ spanned by the values of all possible basic commutators $X_{[I]}$ of order $|I|\leq m$ in $x$:
\[
H^k_x=\langle X_{[I]}(x), |I|\leq k\rangle. 
\]
Denote by $\mathfrak g_{d,m}$ the free nilpotent Lie algebra of step $m$ with $d$ generators, that is, the quotient Lie algebra of the free Lie algebra with $d$ generators $\mathfrak g_d$ by the ideal $\mathfrak g^{m+1}_d$, where $\mathfrak g^{m+1}_d$ is the $(m+1)$-th term of the lower central series of the Lie algebra $\mathfrak g_d$ defined in the following way: $\mathfrak g^1_d=\mathfrak g_d$ and $\mathfrak g^{k}_d=[\mathfrak g^{k-1}_d,\mathfrak g_d]$ for any $k>1$. It is easy to see that $\dim H^k_x\leq \dim \mathfrak g_{d,k}$ for any $x\in U$.  

\begin{definition}
Vector fields $X_1,\ldots,X_d$ are free of order $m$ at $x\in U$, if $\dim H^m_x= \dim \mathfrak g_{d,m}$. 
\end{definition}

One can give another, more detailed definition of a free system of vector fields, following the paper \cite{Hoermander-Melin}. For any set $I=(i_1,\ldots, i_k)$, define a differential operator of order $k$ on $U$ by the formula
\[
X_I=X_{i_1}\ldots X_{i_k}.
\]
Using only the definition of the commutator, any vector field $X_{[I]}$ can be written as a linear combination of differential operators $X_I$:
\[  
X_{[I]}=\sum_{J}A_{IJ}X_J.
\]
For instance, for $|I|=|J|=1$, we have $A_{IJ}=\delta_{i_1j_1}$, for $|I|=|J|=2$, $A_{IJ}=\delta_{i_1j_1}\delta_{i_2j_2}-\delta_{i_1j_2}\delta_{i_2j_1}$ and so on. It is clear that $A_{IJ}=0$, if $|I|\neq |J|$.  

The properties of the Lie bracket: anti-commutativity  
\[
\operatorname{ad} X_{i_{1}} X_{i_2}+\operatorname{ad} X_{i_2} X_{i_1}=0, 
\]
and the Jacobi identity
\[ \operatorname{ad} X_{i_{1}} \operatorname{ad} X_{i_{2}} X_{i_3}+\operatorname{ad} X_{i_{3}} \operatorname{ad} X_{i_{1}} X_{i_2}+\operatorname{ad} X_{i_{2}} \operatorname{ad} X_{i_{3}} X_{i_1}=0,
\]   
determine nontrivial linear relations between the coefficients $A_{IJ}$:
\[
\sum_{I}a_IA_{IJ}=0 \quad \text{for any}\ J.
\]
For instance, for $|I|=|J|=2$, we have $A_{i_1i_2J}+A_{i_2i_1J}=0$. Any linear relation of such a kind gives rise to a linear relation between vector fields $X_{[I]}$:
\[
\sum_{I}a_IX_{[I]}=0.
\]
For instance, for $|I|=|J|=2$, we have $X_{[i_1i_2]}+X_{[i_2i_1]}=0$. 

The vector fields $X_1,\ldots, X_d$ are free of order $m$ at $x\in U$ if and only if any linear relation between $X_{[I]}(x)$ has the above form:
\[
\sum_{|I|\leq m}a_IX_{[I]}(x)=0\Rightarrow \sum_{|I|\leq m}a_IA_{IJ}=0.
\]

\subsection{The lifting theorem}\label{s:lifting}
Fix $m\in M$. Let $\Omega_0$ be a neighborhood of $m$ on which a local orthonormal frame $X_1,\ldots, X_d$ in $H$ is defined. One can assume that $\Omega_0$ is a foliated coordinate neighborhood $\Omega_0\cong U_0\times T_0$, $U_0\subset \mathbb R^p$,  $T_0\subset \mathbb R^q$, and $m$ corresponds to $(0,0)\in U_0\times T_0$. For any $r\geq 1$ and $(x,y)\in U_0\times T_0$, we denote by $H^r_{(x,y)}$ the subspace of $\mathbb R^p$ spanned by the values of all possible basic commutators $X_{[I]}$ of order $|I|\leq r$ at $(x,y)$. Recall that, by assumption, for any $y\in T_0$, the family $\{X_{1,y},\ldots,X_{d,y}\}$ of smooth vector fields in $U_0\subset \mathbb R^p$ satisfies H\"ormander's condition: for some $m$, the subspace $H^m_{(x,y)}$ coincides with $\mathbb R^p$ for any $x\in U_0$.

A generalization of the Rothschild-Stein theorem on lifting of vector fields, which we need, is formulated in the following way. 

\begin{theorem} \label{t:lifting}
Assume that for some $m$ the subspace $H^m_{(0,0)}$ coincides with $\mathbb R^p$. Then there exist a neighborhood $U$ of the origin in $U_0$, a neighborhood $T$ of teh origin in $T_0$, a neighborhood $U^\prime$ of the origin in $\mathbb R^k$, $k= \dim \mathfrak g_{d,m}-p$, and smooth vector field $\tilde X_1,\ldots , \tilde X_d$ on $U\times U^\prime\times T$ of the form
\begin{equation}\label{e:lift}
\tilde X_{i}(x,x^\prime,y)= X_{i}(x,y)+\sum_{j=1}^k
u_{ij}(x,x^\prime,y)\frac{\partial}{\partial x^\prime_j},
\end{equation}
such that, at each point $(x,x^\prime,y)\in U\times U^\prime\times T$, the vector fields $\tilde X_1,\ldots , \tilde X_d$ are free of order $m$ and the subspace $\tilde H^m_{(x,x^\prime,y)}$ coincides with $\mathbb R^{p+k}\subset \mathbb R^{p+k}\oplus \mathbb R^q$. 
\end{theorem}

\begin{proof}
We will give the proof of Theorem \ref{t:lifting} by a slight generalization of the proof of the Rothschild-Stein theorem given in \cite{Hoermander-Melin}. A crucial role in \cite{Hoermander-Melin} is played by the following statement. 

\begin{prop}[\cite{Hoermander-Melin}, Proposition 3]\label{p:lift} Let $X_{1},\ldots, X_{d}$ be a family of smooth vector fields defined in a domain $U\subset \mathbb R^N$. Suppose that $X_1,\ldots, X_d$ are free of order $s-1$ at $0$, but not free of order $s$. Then there exist vector fields $\tilde X_{j}$ on $U\times \mathbb R \subset \mathbb R^{N+1}$ of the form
\begin{equation}\label{e:tildeX}
\tilde X_{j}(x, t)= X_{j}(x)+ u_j(x,t)\frac{\partial}{\partial t},\quad j= 1,\ldots, d,
\end{equation}
where $u_j\in C^\infty(U\times \mathbb R)$, such that the vector fields $\tilde X_{j}$ remain free of order $s-1$ at $0$ and, for any $r \geq s$, 
\[
\dim\tilde H^r_{(0,0)}=\dim H^r_{0}+1.
\]
\end{prop}

Recall that the dimension of $H^r_{0}$ is bounded from above by the dimension of the free nilpotent Lie algebra $\mathfrak g_{d,r}$. Assume that the vector fields $X_1,\ldots, X_d$ are not free of order $m$ at the origin, that is, $\dim H^m_{0}=p<\dim \mathfrak g_{d,m}$ (otherwise, $k=0$ and it is sufficient to put $\tilde X_j=X_j$ in order to complete the proof). Then we can apply Proposition \ref{p:lift} for some $s\leq m$ and construct a new system of vector fields $\tilde X_{1},\ldots, \tilde X_d$ such that $\dim\tilde H^m_{(0,0)}=\dim H^m_{0}+1$. After $k$ steps, we obtain a family  of vector fields, free of order $m$ at $0$: $\tilde H^m_{(0,0,0)}=p+k=\dim \mathfrak g_{n,m}$. Since the dimension of $\tilde H^m_{(x,t,y)}$ is a lower semi-continuous function of $(x,t,y)$, the equality $\dim \tilde H^m_{(x,t,y)}=p+k=\dim \mathfrak g_{n,m}$ holds in some neighborhood of $(0,0,0)$. Without loss of generality, we may assume thts this neighborhood has the form $U\times U^\prime\times T$, as claimed in the theorem, and the vector fields $\tilde X_1,\ldots , \tilde X_d$ have the form given by \eqref{e:tildeX}.
\end{proof} 

\subsection{The approximation theorem}
Let $G_{d,m}$ be a connected, simply connected Lie group of the Lie algebra ${\mathfrak g}_{d,m}$. Denote by $Y_1,\ldots, Y_d$ a basis of the Lei algebra ${\mathfrak g}_{d,m}$, identified with left-invariant vector fields on $G_{d,m}$. The next step is to approximate the lifted vector fields $\tilde X_{1},\ldots, \tilde X_{d}$ constructed in Theorem~\ref{t:lifting} by the vector fields $Y_1,\ldots, Y_d$ on $G_{d,m}$. 

Recall that a Lie algebra $\mathfrak g$ is called a graded Lie algebra, if it is represented as a direct sum of linear spaces ${\mathfrak g}_{d,m}=\bigoplus_{j=1}^m V^j$, satisfying the following properties: $[V^j, V^k]\subset  V^{k+j}$, if $k+j\leq m$, and $[V^j, V^k]=\{0\}$, if $k+j>m$. Note that a graded Lie algebra is always nilpotent (of step $m$). In particular,  the exponential map $\exp : \mathfrak g\to G$ maps diffeomorphically a graded Lie algebra $\mathfrak g$ on the correspoding simply connected Lie group $G$.

The Lie algebra ${\mathfrak g}_{d,m}$ is a graded Lie algebra: ${\mathfrak g}_{d,m}=\bigoplus_{j=1}^m V^j$, where $V^j=\mathfrak g^{j}_{d,m}$ is the $j$-th term of the lower central series of $\mathfrak g_{d,m}$. In other words, $V^j$ is the linear subspace spanned by the basic commutators of order $j$

We defined a natural family $\{\delta_t\}_{t\in {\mathbb R}_+}$  of dilations on the Lie algebra ${\mathfrak g}_{d,m}$, setting $\delta_t(Y)=t^{j}Y$ for $Y\in V^j$. The family $\{\delta_t\}$ is a one-parameter automorphism group of ${\mathfrak g}_{d,m}$. By means of the exponential map $\exp : \mathfrak g_{d,m}\to G_{d,m}$, one can lift this group to a one-parameter automorphism group of $G_{d,m}$, which will be also denoted by $\{\delta_t\}$.

A measurable function $f$ on $G_{d,m}$ is called homogeneous of degree $\lambda$, if $f\circ \delta_t =t^\lambda f$. A differential operator $D$ on $G_{d,m}$ is called homogeneous of degree $\lambda$,  if $D(f\circ \delta_t) =t^\lambda (Df)\circ \delta_t$ for any $t>0$. It is easy to see that if $f$ is a homogeneous function of degree $\alpha$ and $D$ is a homogeneous differential operator of degree $\lambda$, then $Df$ is a homogeneous function of degree $\alpha-\lambda$. 

We put $Y_{j1} =Y_j$ for $j=1,\ldots,d$, and, for each $k = 2,\ldots,m$, fix a basis $Y_{1k}, Y_{2k}, \ldots $ of the space $V^k$. All together vectors $Y_{jk}$, $k=1,\ldots,m$, $j=1,\ldots,\dim V^k$ form a basis of ${\mathfrak g}_{d,m}$. This basis determines a linear map $\mathbb R^N\to {\mathfrak g}_{d,m}$, where $N:=\dim {\mathfrak g}_{d,m}$. Taking the composition of this map with the exponential map, we get a coordinate system on $G_{d,m}$:
\[
(u_{jk}) \in \mathbb R^N\Leftrightarrow \exp(\sum u_{jk}Y_{jk})\in G_{d,m}.
\]

Now we return to the lifted vector fields, constructed in Theorem~\ref{t:lifting}. Thus, let $\tilde X_{1},\ldots, \tilde X_d$ be a family of vector fields, defined on $U\times U^\prime\times T$, where $U$ is a neighborhood of the origin in $\mathbb R^p$,  $U^\prime$  is a neighborhood of the origin in $\mathbb R^k$, $k= N-p$ and $T$  is a neighborhood of the origin in $\mathbb R^q$. We will use notation $\tilde U=U\times U^\prime\subset \mathbb R^N$ and $\tilde x=(x,x^\prime)\in \tilde U$. Suppose that the vector fields $\tilde X_{i}$ are free of order $m$ and the corresponding subspace $\tilde H^m_{(\tilde x,y)}$ coincides with $\mathbb R^{N}\subset \mathbb R^{N}\oplus \mathbb R^q$ at each $(\tilde x,y)\in \tilde U\times T$. We can consider each vector field $\tilde X_j$ as a family $\{\tilde X_{j,y}: y\in T\}$ of smooth vector fields on $\tilde U$ parameterized by $y\in T$. For any $y\in T$, the family of vector fields $\tilde X_{1,y},\ldots, \tilde X_{d,y}$ satisfies H\"ormander's condition of step  $m$ and is free of order $m$. Therefore, we can apply the construction of  \cite{Rothschild-Stein} to this system of vector fields for a fixed $y$.
 
For any $(\tilde x,y)\in \tilde U\times T$, put $\tilde X_{j1,y}(\tilde x) = \tilde X_{j,y}(\tilde x)$ for $j=1,\ldots,d$, and, for $k = 2,\ldots,m$, let $\tilde X_{1k,y}(\tilde x), \tilde X_{2k,y}(\tilde x), \ldots $ be the maximal linearly independent subset of commutators of order $k$ in $\tilde x$, which constructed from the vector fields $\tilde X_{j,y}$ exactly in the same way as the subset $Y_{1k}, Y_{2k}, \ldots $ is constructed from the vectors $Y_j$. Thus, all together vectors $\tilde X_{jk,y}(\tilde x)$ form a basis of $\tilde H^m_{(\tilde x,y)}$, which in its turn determines a linear isomorphism $\tilde H^m_{(\tilde x,y)}\cong \mathbb R^N$. For any $y\in T$, we fix a Riemannian metric on $\tilde U$ such that the basis $\{\tilde X_{jk,y}(\tilde x)\}$ is orthonormal at each $(\tilde x,y)\in \tilde U$ and denote by  $dv_y(\tilde x)$ the corresponding volume form on $\tilde U$. 

As above, using the exponential map on the Lie algebra of vector fields, we can construct a coordinate system on $\tilde U$ in a neighborhood of $\tilde x$:
\[
(u_{jk}) \in \mathbb R^N\Leftrightarrow \exp(\sum u_{jk}\tilde X_{jk,y})(\tilde x) \in \tilde U,\quad y\in T.
\]
Here $\{\exp(\tau \sum u_{jk}\tilde X_{jk,y}), \tau\in \mathbb R\}$ denotes the one-parameter family of local diffeomorphisms $\tilde U$, determined by the vector field $\sum u_{jk}\tilde X_{jk,y}$. 
 
Any vector field $X$, defined in the neighborhood $\Omega$ of zero in  $G_{d,m}$, can be written as
\[
X=\sum a_{jk}Y_{jk},\quad  a_{jk}\in C^\infty(\Omega).
\]
If we expand each coefficient $a_{jk}$ in Taylor series in $0$ in the coordinate system on $G_{d,m}$ defined above, we obtain a representation of $X$ as a formal series of homogeneous differential operators. (Observe that if $f$ is a homogeneous function of degree $i$, then $fY_{jk}$ is a homogeneous differential operator of degree $k-i$). We say that $X$ has local degree $\leq \lambda$, if each term of this formal series is a homogeneous differential operator of degree $\leq \lambda$. 

There is the following analogue of the Rothschild-Stein approximation theorem \cite[Theorem 5]{Rothschild-Stein}.
 
\begin{theorem}\label{t:Theta} There are a neighborhood $V$ of the origin in $G_{d,m}$, a neighborhood $\tilde U$ of the origin in $\mathbb R^N$, a neighborhood $T$ of the origin in $\mathbb R^k$ and a family of smooth maps $\Theta_y : \tilde U\times \tilde U \to V$, $y\in T$,  satisfying the following properties:
\begin{description}
\item[(a)] $\Theta_y(\tilde x_1,\tilde x) = \Theta_y(\tilde x,\tilde x_1)^{-1}$, in particular, $\Theta_y(\tilde x,\tilde x)= 0$;
\item[(b)] For every fixed $(\tilde x,y)\in \tilde U\times T$, the map $\Theta_{(\tilde x,y)}: \tilde x_1 \to \Theta_y(\tilde x,\tilde x_1)$ is a diffeomorphism of an neighborhood of $\tilde x$ on a neighborhood of the origin in $G_{d,m}$.
\item[(с)] For every $\tilde x\in \tilde U$ and $j=1,\ldots,d$, we have the equality $(\Theta_{(\tilde x,y)})_*\tilde X_{j,y}=Y_j + R_{j,y}$, where $R_{j,y}$ has local degree $\leq 0$ for any $y\in T$.
\end{description}
\end{theorem}  
 
\begin{proof}
The map $\Theta_y: \tilde U\times \tilde U \to V$ is defined as follows. If $\tilde x_1 =\exp(\sum u_{jk}\tilde X_{jk,y})(\tilde x)$, then $\Theta_y(\tilde x,\tilde x_1)=\exp(\sum u_{jk}Y_{jk})\in V$.  Its smooth dependence on $y$ follows directly from the definitions. All remaining statements, in particular, statement (c), are of algebraic nature and, therefore, can be easily deduced from the corresponding results of \cite{Rothschild-Stein} (cf. \cite[Theorem 5]{Rothschild-Stein}). 
\end{proof}

\subsection{The construction of a parametrix for the lifted operator}
In this section, we describe the parametrix construction for the family of H\"ormander type operators defined by the system of the lifted vector fields $\tilde X_{1},\ldots, \tilde X_d$ in the neighborhood $\tilde U$, given by Theorem \ref{t:Theta}. First of all, we need some information on left-invariant operators on the nilpotent Lie group $G_{d,m}$. 

Fix an inner product on $\mathfrak g_{d,m}$ such that the basis $\{Y_{jk}\}$ in $\mathfrak g_{d,m}$ is orthonormal. By the exponential map, the Lebesgue measure on $\mathfrak g_{d,m}$ corresponds to a left-invariant Haar measure $du$ on $G_{d,m}$. The number $Q = \sum_{j=1}^mj\dim V^j$ is called the homogeneous dimension of $G_{d,m}$. The following relation holds:
\[
d(\delta_tu) = t^Qdu, \quad t > 0.
\]

Fix a homogeneous norm $|\cdot|$ on $G_{d,m}$. We recall that a homogeneous norm on $G_{d,m}$ is a continuous function $u\mapsto|u|$ on $G_{d,m}$ with values in $[0,\infty)$, which is smooth outside of zero, homogeneous of degree one: $|\delta_t(u)|=t|u|, t>0$, and satisfies the conditions: (a) $|u| = |u^{-1}|$ for any $u\in G_{d,m}$, (b) $|u| = 0$ if and only if $u = 0$. The construction of such a function on $G_{d,m}$ is given in \cite{Folland75}. 

For $\lambda > 0$, a kernel of type $\lambda$ is a function on $G_{d,m}$, which is smooth outside the origin and homogeneous of degree $\lambda - Q$:
\[
k(\delta_tu)=t^{\lambda-Q}k(u),\quad t>0. 
\]
Any kernel of type $\lambda$ is a locally integrable function on $G_{d,m}$ and, therefore, defines a distribution on $G_{d,m}$, homogeneous of degree $\lambda$. Similarly, a singular integral kernel is a function $k$ on $G_{d,m}$, which is smooth outside the origin, is homogeneous of degree $-Q$ and, for any $a, b, 0 < a < b < \infty$,  satisfies the condition
\[
\int_{a<|u|<b} k(u)du = 0.
\]
A singular integral kernel $k$ defines a distribution on $G_{d,m}$ by using the principal value integral. A kernel of type zero is called any distribution on $G_{d,m}$ of the form $k + c\delta_0$, where $k$ is a singular integral kernel, $\delta_0$ is the delta-function at the origin and $c$ is some complex number. 

Now we consider the neighborhoods $\tilde U$ and $T$ of the origin in $\mathbb R^N$ and $\mathbb R^q$ respectively, given by Theorem \ref{t:Theta}. For $\lambda\geq 0$, a kernel of type $\lambda$ is called a family $\{K_y : y\in T\}$ of functions on $\tilde U\times \tilde U$, which can be written for any $l\in \mathbb N$ as 
\[
K_y(\tilde x,\tilde x_1)=\sum_{i=1}^s a_i(\tilde x,y)k^{(i)}_{\tilde x,y}(\Theta_y(\tilde x_1,\tilde x)) b_i(\tilde x_1,y)+E_{l,y}(\tilde x,\tilde x_1), \quad (\tilde x,\tilde x_1)\in \tilde U\times \tilde U,
\] 
where $E_{l,y}\in C^l(\overline{\tilde U}\times \overline{\tilde U})$ and $a_{i,y}, b_{i,y}\in C^\infty_c(\tilde U)$, $i=1,\ldots,s$ depend smoothly on $y$ and, for any $i$, the function $u\in V \mapsto k^{(i)}_{\tilde x,y}(u)$ is a kernel of type $\geq \lambda$, depending smoothly on $(\tilde x,y)$.

A kernel $\{K_y : y\in T\}$ of type $\lambda\geq 0$ defines a family of operators $\{P_y: C^\infty_c(\tilde U) \to C^\infty(\tilde U) : y\in T\}$ by
\[
P_yf(\tilde x)=\int_{\tilde U}K_y(\tilde x, \tilde x_1) f(\tilde x_1)dv_y(\tilde x_1), \quad f\in C^\infty_c(\tilde U).
\]
The family $\{P_y: y\in T\}$ will be called a family of operators of type $\lambda$ on $\tilde U$.

\begin{prop}\label{p:L2}
A family $\{P_y : y\in T\}$ of operators of type $\lambda \leq 0$ on $\tilde U$ defines a continuous in uniform operator topology family of bounded operators in $L^2(\tilde U)$.
\end{prop}

\begin{proof}
We observe that if $K$ is a kernel of type $\lambda$, then 
\[
|K_y(\tilde x, \tilde x_1)|\leq C(\rho_y(\tilde x, \tilde x_1))^{-Q+\lambda},
\]
where $\rho_y$ is the pseudo-metric given by $\rho_y(\tilde x, \tilde x_1)=|\Theta_y(\tilde x, \tilde x_1)|$. Therefore, for $\lambda>0$, we have 
\[
\int_{\rho_y(\tilde x, \tilde x_1)<C}|K_y(\tilde x, \tilde x_1)|dv_y(\tilde x_1) \leq C\int_{|u|<C}|u|^{-Q+\lambda}du<\infty,
\]
which easily implies the desired statement for $\lambda>0$. Its proof in the case $\lambda=0$ is given by a rather straightforward generalization of the proof of Theorem 6 in \cite{Rothschild-Stein}.
\end{proof}

Let $\tilde X_{1},\ldots, \tilde X_{n}$ be the family of vector fields on $\tilde U\times T$, constructed in Theorem~\ref{t:lifting}. We will consider every vector field $\tilde X_j$ as a smooth family $\{\tilde X_{j,y}: y\in T\}$ of vector fields on $\tilde U$. For any $y\in T$, the system of vector fields $\tilde X_{1,y},\ldots, \tilde X_{d,y}$ satisfies the H\"ormander condition of step $m$ and free of order $m$. Consider the family 
\begin{equation}\label{e:tildeLy}
\tilde L_y = \sum_{j=1}^d\tilde X_{j,y}^2.
\end{equation}
of H\"ormander sum of squares type operators in the domain $\tilde U$, smoothly depending on $y\in T$.  

\begin{theorem}\label{t:tilde-param}
For any $\psi\in C^\infty_c(\tilde U\times T)$, there exists a family $\{\tilde P_y : y\in T\}$ of operators of type $2$ on $\tilde U$, compactly supported in $\tilde U\times T$, such that 
\[
\tilde L_y\tilde P_y=\psi_y I-\tilde R_y,\quad \tilde P_y\tilde L_y=\psi_yI -\tilde S_y,\quad y\in T,
\]
where $\{\tilde R_y\}$ and $\{\tilde S_y\}$ are families of operators of type 1 on $\tilde U$, compactly supported in $\tilde U\times T$. 
\end{theorem}

\begin{proof}
Consider the left-invariant hypoelliptic second order differential operator  $L_0 = \sum_{j=1}^dY_j^2$ on the Lie group $G_{d,m}$. By \cite{Folland75}, there exists a unique distribution $K_0\in C^{-\infty}(G_{d,m})$, which is a fundamental solution for $L_0$:
\[
L_0(K_0)=\delta_0. 
\]
Moreover, $K_0$ is a kernel of type $2$. 

Let $\psi^\prime \in C^\infty_c(\tilde U\times T)$ be a function such that $\psi^\prime=1$ on the support of $\psi$. Define a kernel $\{\tilde K_y\}$ of type 2 on $\tilde U\times \tilde U$ by
\begin{equation}\label{e:tildeK}
\tilde K_y(\tilde x,\tilde x_1)=\psi(\tilde x,y)K_0(\Theta_y(\tilde x_1,\tilde x))\psi^\prime(\tilde x_1,y).
\end{equation}
Let $\tilde P_y$ be the corresponding family of operators of type 2 on $\tilde U$. For any $y\in T$, the operator $\tilde P_y$ is a parametrix for $\tilde L_y$ (см. \cite[Theorem 10]{Rothschild-Stein}). Smooth dependence of families $\{\tilde R_y\}$ and $\{\tilde S_y\}$ on $y$ can be easily checked, taking into account smooth dependence of $\Theta_y$ on $y$. 
\end{proof}

For any $k\in \mathbb Z_+$ and $y\in T$, we introduce the space $S^2_{k,y}(\tilde U)$, which consists of all functions $f\in L^2(\tilde U)$ such that $X_{I,y}f\in L^2(\tilde U)$ for any $I$ with $|I|\leq k$. Define a norm in this space by 
\[
\|f\|_{S^2_{k,y}}=\sum_{|I|\leq k}\|X_{I,y}f\|_{L^2}.
\]

An operator family $\{P_y : C^\infty(\overline{\tilde U})\to C^\infty(\overline{\tilde U}) : y\in T\}$ is said to be a continuous family of smoothing operators of order $\lambda\in \mathbb Z_+$, if, for any $y\in T$, the operator $P_y$ extends to a bounded operator from $S^2_{k,y}(\tilde U)$ to $S^2_{k+\lambda,y}(\tilde U)$, and, moreover, its norm is uniformly bounded in $y\in T$: 
\[
\sup_{y\in T}\|P_y: S^2_{k,y}(\tilde U)\to S^2_{k+\lambda,y}(\tilde U)\|<\infty. 
\]
As an immediate consequence of Proposition \ref{p:L2}, we get the following assertion (cf. also \cite[Theorem 11]{Rothschild-Stein}).

\begin{prop}\label{type-smooth}
Any family $\{P_y\}$ of operators of type $\lambda$ is a continuous family of smoothing operators of order $\lambda$.
\end{prop}

 \subsection{Push-forward of a parametrix}
In this section, we construct a parametrix for a family of H\"ormander type operators, defined by a system of vector fields $X_{1},\ldots, X_d$, using the parametrix constructed in the previous section and a push-forward construction.

Recall that $\tilde U=U\times U^\prime$, where $U$ is a neighborhood of $0$ in $\mathbb R^p$ and $U^\prime$ is a neighborhood of $0$ in $\mathbb R^k$. We introduce the pull-back operator $E: C^\infty(\overline{U})\to C^\infty(\overline{\tilde U})$ given for $f\in C^\infty(\overline{U})$ by
\[
Ef(x, x^\prime) = f(x), \quad (x, x^\prime)\in U\times U^\prime.
\] 

We also introduce the restriction operator $R:C^\infty(\overline{\tilde  U})\to C^\infty(\overline{U})$ as follows. Fix a function $\zeta\in C^\infty_c(U^\prime)$ such that $\int \zeta(x^\prime)dx^\prime=1$, where $dx^\prime$ is the Lebesgue measure on $\mathbb R^k$. Then, for $f\in C^\infty(\overline{\tilde  U})$, we put 
\[
Rf(x)=\int_{U^\prime} f(x,x^\prime)\zeta(x^\prime)dx^\prime, \quad x\in U.
\]
The operator $R$ is a left inverse of $E$:
\[
R\circ E=I : C^\infty(\overline{U})\to C^\infty(\overline{U}). 
\]

As above, one can introduce the spaces $S^2_{k,y}(U)$ and the notion of continuous family of smoothing operators or order $\lambda$ on $U$.

\begin{prop}\label{t:ES}
For any $y\in T$, the map $E$ defines a continuous map from $S^2_{k,y}(U)$ to $S^2_{k,y}(\tilde U)$ with the norm, uniformly bounded on $y\in T$.
\end{prop}

\begin{proof}
First, we observe that $E$ defines a continuous map from $L^2(U)$ to $L^2(\tilde U, dv_y)$, uniformly bounded on $y\in T$. From the definition of $E$ and the fact that the vector fields $\tilde X_{j,y}$ are the lifts of the vector fields $X_{j,y}$ (см. \eqref{e:tildeX}), it follows that 
\[
\tilde X_{I,y}E = EX_{I,y}, \quad y\in T,
\]
that immediately implies the desired statement.
\end{proof}

\begin{prop}\label{t:RS}
For any $y\in T$, the map $R$ defines a continuous map from $S^2_{k,y}(\tilde U)$ to $S^2_{k,y}(U)$ with the norm, uniformly bounded on $y\in T$.
\end{prop}

\begin{proof}
First, we observe that an arbitrary operator $R^\prime :C^\infty(\overline{\tilde  U})\to C^\infty(\overline{U})$, defined for $f\in C^\infty(\overline{\tilde  U})$ by
\begin{equation}
\label{e:R-prime}
R^\prime f(x)=\int_{U^\prime}\lambda(x,x^\prime) f(x,x^\prime)dx^\prime, \quad x\in U,
\end{equation}
where $\lambda\in C^\infty(\overline{\tilde  U})$ is some function, defines a continuous map from $L^2(\tilde U)$ to $L^2(U)$. In particular, this holds for the operator $R$ itself.

Now we show that, for any operator $R^\prime$ of the form \eqref{e:R-prime}, the following formula holds:
\begin{equation}
\label{e:XkR}
X_{j,y}R^\prime= R^\prime \tilde X_{j,y} + R_{j,y}^\prime,\quad y\in T,
\end{equation}
where $\{R^\prime_{j,y}, y\in T\}$ is a smooth family of operators of the form \eqref{e:R-prime}. 

Indeed, by \eqref{e:lift}, for $f\in C^\infty(\overline{\tilde  U})$, we have
\[
R^\prime \tilde X_{j,y} f(x)=\int_{U^\prime}\lambda(x,x^\prime) \left(X_{j,y}(x)f(x,x^\prime)+\sum_{\ell=1}^k u_{j\ell,y}(x,x^\prime)\frac{\partial f}{\partial x^\prime_\ell}(x,x^\prime)\right)dx^\prime, \quad x\in U.
\]
Integrating by parts, we get the formula \eqref{e:XkR} with 
\[
R^\prime_{j,y}f(x)=-\int_{U^\prime} \left(X_{j,y}(x)\lambda(x,x^\prime)+\sum_{\ell=1}^k \frac{\partial }{\partial x^\prime_\ell}(u_{j\ell,y}(x, x^\prime)\lambda(x,x^\prime))\right) f(x,x^\prime)dx^\prime, \quad x\in U.
\]

By repeated application of the formula \eqref{e:XkR}, we obtain that, for any $I$, the formula holds:
\[
X_{I,y}R = R \tilde X_{I,y} +\sum_{|J|<|I|} R^\prime_{J,y}X_{J,y},
\]
where $\{R^\prime_{J,y}, y\in T\}$ are smooth families of operators of the form \eqref{e:R-prime}, that immediately completes the proof. 
\end{proof}

To any operator family $\{\tilde P_y : C^\infty(\overline{\tilde  U}) \to C^\infty(\overline{\tilde  U}), y\in T\}$, one can assign an operator family $\{P_y : C^\infty(\overline{U}) \to C^\infty(\overline{U}), y\in T\}$ by the formula
\[
P_y = R\circ \tilde P_y\circ E, \quad y\in T.
\]
If $\tilde K_y(\tilde x,\tilde x_1)$ is the kernel of the operator $\tilde P_y$, then the kernel $K_y(x, x_1)$ of the operator $P_y$ is given by
\[
K_y(x, x_1)=\int_{U^\prime\times U^\prime} \zeta(x^\prime)\tilde K_y((x,x^\prime), (x_1,x^\prime_1))\frac{dv_y}{dx_1}(x_1) dx^\prime\,dx^\prime_1.
\]
In particular, if the family $\{\tilde P_y, y\in T\}$ is compactly supported in $\tilde U\times T$, then the family $\{P_y, y\in T\}$ is compactly supported in $U\times T$.

\begin{prop}\label{p:tildeT-T}
If $\{\tilde P_y\}$ is a continuous family of smoothing operators of order $\lambda$ on $\tilde U$, then $\{P_y\}$ is a continuous family of smoothing operators or order $\lambda$ on $U$. 
\end{prop}

\begin{proof}
The statement is a simple consequence of Propositions~\ref{t:ES} and \ref{t:RS}. 
\end{proof}

Consider a smooth family of H\"ormander sum of squares type operators in  $U$ given by
\begin{equation}\label{e:Ly}
L_y=\sum_{j=1}^d X^2_{j,y}, \quad y\in T,
\end{equation}
The operator $\tilde L_y$ on $\tilde U$ given by \eqref{e:tildeLy} is the lift of the operator $L_y$: 
\[
L_y = R\tilde L_yE, \quad y\in T.
\]

\begin{theorem}\label{t:paramLy}
For any $\phi\in C^\infty_c(U\times T)$, there exists a smooth family $\{P_y\}$ of smoothing operators of order $2$ on $U$, compactly supported in $U\times T$ such that 
\begin{equation}\label{e:LP-PL}
L_yP_y=\phi_y I-R_y,\quad P_yL_y=\phi_y I-S_y,
\end{equation}
where $\{R_y\}$ and $\{S_y\}$ are smooth families of smoothing operators of order 1 on $U$, compactly supported in $U\times T$. 
\end{theorem}

\begin{proof}
For a given $\phi\in C^\infty_c(U\times T)$, we define a function $\psi\in C^\infty_c(\tilde U\times T)$ by $\psi(x,x^\prime,y) = \phi(x,y)\zeta^\prime(x^\prime)$, where $\zeta^\prime \in C^\infty_c(U^\prime)$ is a function such that $\zeta^\prime=1$ on the support of $\zeta$. By Theorem~\ref{t:tilde-param}, there exists a smooth family $\{\tilde P_y\}$ of integral operators of type $2$ on $\tilde U$ such that 
\begin{equation}\label{e:tildeLP-PL}
\tilde L_y\tilde P_y=\psi_y I-\tilde R_y,\quad \tilde P_y\tilde L_y=\psi_y I-\tilde S_y,
\end{equation}
where $\{\tilde R_y\}$ and $\{\tilde S_y\}$ are smooth families of integral operators of type 1 on $\tilde U$. 

For any $y\in T$, we define an operator $P_y$ on $U$ by
\[
P_y=R\tilde P_yE.
\]
The relations \eqref{e:tildeLP-PL} and the properties of the operators $E$ and $R$ imply the validity of the relations \eqref{e:LP-PL} with $
R_y=R\tilde R_yE$ and $P_y=R\tilde P_yE$, and the proof is completed with the use of Proposition \ref{p:tildeT-T}.
\end{proof}

\subsection{Proof of Theorem~\ref{t:multi2}} In this section, we complete the proof of Theorem~\ref{t:param-local}. As shown in Section~\ref{s:local},  thus we complete the proof of Theorem~\ref{t:param}, which, as shown in Section \ref{s:regular}, in its turn, implies Theorem~\ref{t:multi2}. 

Let us briefly recall our arguments. Let $m\in M$. As in Section \ref{s:lifting}, let $\Omega_0$ be a neighborhood of $m$, on which a local orthonormal frame $X_1,\ldots, X_d$ in $H$ is defined. We assume that $\Omega_0$ is a foliated coordinate neighborhood, $\Omega_0\cong U_0\times T_0$, $U_0\subset \mathbb R^p$,  $T_0\subset \mathbb R^q$, and $m$ corresponds to $(0,0)\in U_0\times T_0$. In Theorem~\ref{t:lifting}, we constructed a system $\tilde X_1,\ldots, \tilde X_d$ of the lifted vector fields, defined in a neighborhood $\tilde U\times T$, where $\tilde U=U\times U^\prime\subset \mathbb R^N$, $U$ is a neighborhood of the origin in $U_0$, $T$ is a neighborhood of the origin in $T_0$, $U^\prime$ is a neighborhood of the origin in $\mathbb R^k$. Choose $U$ and $\tilde U$ such that the assertion of Theorem \ref{t:Theta} holds. Then, in the neighborhoods $U\times T$ and $\tilde U\times T$, one can construct parametrices for the operators $L$ and $\tilde L$, respectively. In particular, in the foliated coordinate neighborhood $\Omega\cong U\times T$, the assertion of Theorem \ref{t:paramLy} holds. It is in this neighborhood that the validity of Theorem~\ref{t:param-local} will be proved.  

Recall (see \eqref{e:Delta-local}) that, in the foliated coordinate neighborhood $\Omega\cong U\times T$, the operator family 
$\{\Delta_y: y\in T\}$ has the form
\[
\Delta_y=-L_y+\sum_{j=1}^da_{j,y}X_{j,y}+b_y, \quad y\in T, 
\]
where $a_j,b\in C^\infty(U\times T)$. Therefore, Theorem \ref{t:paramLy} immediately implies the following statement. 

\begin{prop}
For any $\phi\in C^\infty_c(U\times T)$, there exists a smooth family $\{P_y\}$ of smoothing operators of order $2$ on $U$, compactly supported in  $U\times T$, such that 
\[
\Delta_yP_y=\phi_y I-R^\prime_y,\quad P_y\Delta_y=\phi_y I-S^\prime_y,
\]
where $\{R^\prime_y\}$, $\{S^\prime_y\}$ are smooth families of smoothing operators of order $1$ on $U$, compactly supported in $U\times T$.
\end{prop}

Standard iterative arguments (cf. \cite[Corollary 17.14]{Rothschild-Stein}) allow one to improve the parametrix. Since these arguments are purely algebraic, they can be easily extended to smooth operator families.  

\begin{prop}\label{p:Pl}
For any $\phi\in C^\infty_c(U\times T)$ and $l\in \mathbb N$, there exists a smooth family $\{P_{l,y}\}$ of smoothing operators of order $l$ on $U$, compactly supported in $U\times T$, such that 
\[
\Delta_yP_{l,y}=\phi_y I-R_{l,y},\quad P_{l,y}\Delta_y=\phi_y I-S_{l,y},
\]
where $\{R_{l,y}\}$, $\{S_{l,y}\}$ are smooth families of smoothing operators of order $l$ on $U$, compactly supported in $U\times T$.
\end{prop}

Now we are ready to complete the proof of Theorem~\ref{t:param-local}. Let us consider the standard Sobolev spaces $L^2_k(U)$, which consist of all functions $f\in L^2(U)$ such that $\frac{\partial^{|\alpha|}f}{\partial x_1^{\alpha}\ldots\partial x_p^{\alpha_p}}\in L^2(U)$ for every multi-index $\alpha\in \mathbb Z^p_+$ wtih $|\alpha|\leq k$, equipped with the norm
\[
\|f\|^2_{L^2_k}=\sum_{|\alpha|\leq k} \left\|\frac{\partial^{|\alpha|}f}{\partial x_1^{\alpha}\ldots\partial x_p^{\alpha_p}} \right\|^2_{L^2}.
\]
It is easy to see that, for any $k\in \mathbb N$ and $y\in T$, there is an inclusion  
\begin{equation}\label{e:embed}
S^2_{km,y}(U)\subset L^2_k(U).  
\end{equation}
The embedding operators \eqref{e:embed} are uniformly bounded in $y\in T$: for any $k\in \mathbb N$, there exists a constant $C_k>0$ such that 
\begin{equation}\label{e:embed1}
\|f\|_{L^2_k}\leq C_k\|f\|_{S^2_{km,y}}, \quad f\in S^2_{km,y}(U), \quad y\in T.  
\end{equation}
Indeed, by assumption, any vector field $\frac{\partial}{\partial x_j}$ can be represented as a linear combination of vector fields $X_{[I],y}$ with $|I|\leq m$, and, hence, as a linear combination of differential operators  $X_{I,y}$ with $|I|\leq m$. The coefficients of these linear combinations are smooth functions on $U$, smoothly depending on $y\in T$. Therefore, any differential operator $\frac{\partial^{|\alpha|}}{\partial x_1^{\alpha}\ldots\partial x_p^{\alpha_p}}$ can be represented as a linear combinations of differential operators $X_{I}$ с $|I|\leq m |\alpha|$, whose coefficients are smooth functions on $U$, smoothly depending on $y\in T$. This easily implies the estimate  \eqref{e:embed1}.

Let us show that, for sufficiently large $l$, the family $\{Q_y=P_{l,y}: y\in T\}$ is the desired family of parametrices that allows us to prove Theorem~\ref{t:param-local} c $R_y=R_{l,y}, S_y=S_{l,y}$. 

Let $l>m+2$. Then each operator family $R_y$, $S_y$, $\Delta_y R_y$ and $\Delta_y S^*_y$ is a family of bounded operators from $L^2(U)$ to $S^2_{l,y}(U)$ with the norm, uniformly bounded on $y\in T$. Taking the compositions of these families with the uniformly bounded family of the embedding operators \eqref{e:embed}, we obtain families of bounded operators from $L^2(U)$ to $L^2_1(U)$ with the norm, uniformly bounded in $y\in T$. Since each of these operators maps the space $C^\infty(\bar U)$ to itself, one can show that these families of bounded operators from  $L^2(U)$ to $L^2_1(U)$ are continuous in the strong operator topology of the space $\mathcal L(L^2(U), L^2_1(U))$. Taking the compositions of the operators with the compact embedding operator $L^2_1(U)\hookrightarrow L^2(U)$, we conclude that all the operators under consideration are compact in $L^2(U)$ and the corresponding families of bounded operators in $L^2(U)$ are continuous in the uniform operator topology of $\mathcal L(L^2(U))$.   

For the family $Q_y$, the proof is obtained with similar arguments, using a more refined assertion that, for some $\varepsilon>0$, the uniformly bounded on $y$ embedding $S^2_{1,y}(U)\subset L^2_\varepsilon(U)$ holds. This assertion is a particular case of the subelliptic estimates established in  \cite{hypo} (see Theorem \ref{t:Hs-hypo}). It can be also deduced by generalizing the proof of Theorem 13 from \cite{Rothschild-Stein} to the families case. Thus, the proof of Theorem~\ref{t:param-local} is completed.

\end{document}